\documentclass{compositio}

\usepackage{amsmath, amscd, amssymb}

\newtheorem{thm}{Theorem}[section]
\newtheorem{cor}[thm]{Corollary}
\newtheorem{prop}[thm]{Proposition}
\newtheorem{lemma}[thm]{Lemma}

\theoremstyle{remark}
\newtheorem{remark}[thm]{Remark}

\newtheorem*{notation}{Notation}

\theoremstyle{definition}
\newtheorem{definition}[thm]{Definition}
\newtheorem{conj}[thm]{Conjecture}

\newcommand{\op}{\mathtt{op}}
\newcommand{\id}{{\mathtt{Id}}}
\newcommand{\dbar}{\overline{\partial}}
\newcommand\dual[1]{{#1}^{\vee}}
\newcommand{\Conn}{\mathbb{\nabla}}
\newcommand{\poisson}{\pi}
\newcommand{\EXT}{\mathcal{E\!\scriptstyle{XT}}}

\DeclareMathOperator{\Ad}{Ad}
\DeclareMathOperator{\gr}{gr}
\DeclareMathOperator{\DQ}{\mathfrak{DQ}}
\DeclareMathOperator{\DQA}{DQ-alg}
\DeclareMathOperator{\CONN}{Conn}
\DeclareMathOperator{\Hom}{Hom}
\DeclareMathOperator{\Pic}{Pic}
\DeclareMathOperator{\shHom}{\underline{Hom}}
\DeclareMathOperator{\shEnd}{\underline{End}}

\DeclareMathOperator{\shPic}{\underline{Pic}}
\DeclareMathOperator{\rk}{rk}
\DeclareMathOperator{\ad}{ad}
\DeclareMathOperator{\pr}{pr}

\begin{document}
\title{On quasi-classical limits of DQ-algebroids}

\author[P.Bressler]{Paul Bressler}
\email{paul.bressler@gmail.com}
\address {Departamento de Matem\'aticas, Universidad de Los Andes, Bogot\'a, Colombia}

\author[A.Gorokhovsky]{Alexander Gorokhovsky}
\email{Alexander.Gorokhovsky@colorado.edu}
\address{Department of Mathematics, UCB 395,
University of Colorado, Boulder, CO~80309-0395, USA}

\author[R.Nest]{Ryszard Nest}
\email{rnest@math.ku.dk}
\address{Department of Mathematics,
Copenhagen University, Universitetsparken 5, 2100 Copenhagen, Denmark}

\author[B.Tsygan]{Boris Tsygan}
\email{b-tsygan@northwestern.edu}
\address{Department of
Mathematics, Northwestern University, Evanston, IL 60208-2730, USA}



\thanks{A.~Gorokhovsky was partially supported by NSF grant DMS-0900968. B.~Tsygan was partially
supported by NSF grant DMS-0906391. R.~Nest was supported by the Danish National Research Foundation through the Centre for Symmetry and Deformation (DNRF92).}

\begin{abstract}
We determine the additional structure which arises on the classical limit of a DQ-algebroid.
\end{abstract}

\maketitle

\section{Introduction}
As is well-known, a one-parameter formal deformation of a commutative algebra gives rise to a Poisson bracket on the latter. The resulting Poisson algebra is usually referred to as the quasi-classical limit of the deformation. This circumstance is responsible for the link between deformation quantization and Poisson geometry. Although deformation quantization had been developed in the context of (sheaves of) algebras and deformations thereof, it became apparent from the work of M.~Kontsevich (\cite{K01}) and M.~Kashiwara (\cite{Kash}) that the broader context of algebroid stacks provides a more natural setting for the theory leading to the notion of a DQ-algebroid due to M.~Kashiwara and P.~Schapira (\cite{KS}).

Unlike the more traditional setting, the classical limit of a DQ-algebroid on a manifold $X$ is not the structure sheaf $\mathcal{O}_X$\footnote{Here and below we consider complex valued functions.} but, rather, a twisted form thereof which is, in essence, a linear version of a $\mathcal{O}^\times_X$-gerbe. The purpose of the present note is to determine the additional structure which arises on the classical limit of a DQ-algebroid and reflects the non-commutative aspect of the latter in the way in which the Poisson bracket reflects non-commutativity of the deformation quantization of $\mathcal{O}_X$. We shall presently describe our answer.

A DQ-algebroid $\mathcal{C}$ gives rise to \emph{the associated Poisson bi-vector $\pi^\mathcal{C}$} on $X$, hence to a Lie algebroid structure on the cotangent sheaf $\Omega^1_X$ which we denote $\Pi^\mathcal{C}$. The notion of a connective structure with curving on a $\mathcal{O}^\times_X$-gerbe readily generalizes to the notion of a $\Pi^\mathcal{C}$-connective structure with curving (and, in fact, to the notion of a $\mathcal{B}$-connective structure with curving for any Lie algebroid $\mathcal{B}$). The principal result of this note is the construction of a canonical $\Pi^\mathcal{C}$-connective structure with flat curving on the classical limit of the DQ-algebroid $\mathcal{C}$. As an application we show (Theorem \ref{thm: obst and gemuese}) that an $\mathcal{O}^\times_X$-gerbe $\mathcal{S}$ admits a deformation with the associated Poisson bi-vector $\pi$ only if the class $[\mathcal{S}]\in H^2(X;\mathcal{O}^\times_X)$ lifts to Deligne--Poisson cohomology $H^2(X;\mathcal{O}^\times_X\to\Omega^1_\Pi\to\Omega^{2,cl}_\Pi)$, where $\Pi$ is the Lie algebroid (structure on $\Omega^1_X$) determined by $\pi$. Similar ideas have been used in \cite{BGKP} to give a necessary and sufficient condition for a line bundle on a lagrangian subvariety of a symplectic manifold to admit a quantization.

In \cite{B02} (see also \cite{BW04}), the canonical flat contravariant connection on the classical limit of a bi-module over a pair of star-product algebras is constructed. Our variant of this construction (see \ref{subsection: Quasi-classical limits of bi-modules}) is a key component of the present work. 

We work in the natural generality of a $C^\infty$ manifold $X$ equipped with an integrable complex distribution $\mathcal{P}$ satisfying additional technical conditions. Subsumed thereby are the case of plain $C^\infty$ manifold $X$ ($\mathcal{P}$ trivial) as well as $X$ a complex analytic manifold ($\mathcal{P}$ a complex structure).

The paper is organized as follows. Relevant facts about calculus in presence of an integrable distribution are recalled in Section \ref{section: distributions}. Section \ref{section: algebroid stacks} contains basic facts about torsors and gerbes and establishes notation. In Section \ref{section: lie algebroids} we review the theory of Lie algebroids and modules. In particular, we describe the cohomological classification of invertible modules over a Lie algebroid which may be of independent interest. Section \ref{section: connective structures} is devoted to generalization of the notions of connective structures and curving to the Lie algebroid setting. In Section \ref{section: DQ-algebras} we review basic facts on deformation quantization algebras and bi-modules and define the \emph{subprincipal curvature}. Section \ref{section: qcl} contains the necessary background material on DQ-algebroids and the principal result of the present note -- the construction of the quasi-classical limit of a DQ-algebroid. In Section \ref{section: formality} we discuss the relationship between the construction of the quasi-classical limit and the quasi-classical description of \cite{BGNT1} of the category of DQ-algebroids provided by the formality theorem. Section \ref{section: formality} concludes with a conjecture on the precise nature of the aforementioned relationship.

It is the authors' pleasure to express their gratitude to the referee for thorough reading of the original manuscript and thoughtful comments and suggestions.

\section{Calculus in presence of an integrable distribution}\label{section: distributions}
In this section we briefly review basic facts regarding differential calculus in the presence of an integrable complex distribution. We refer the reader to \cite{Kostant}, \cite{Rawnsley} and \cite{FW} for details and proofs.

For a $C^\infty$ manifold $X$ we denote by $\mathcal{O}_X$ (respectively, $\Omega^i_X$) the sheaf of \emph{complex valued} $C^\infty$ functions (respectively, differential forms of degree $i$) on $X$. Throughout this section we denote by $\mathcal{T}_X^\mathbb{R}$ the sheaf of \emph{real valued} vector fields on $X$. Let $\mathcal{T}_X := \mathcal{T}_X^\mathbb{R}\otimes_\mathbb{R}\mathbb{C}$.

\subsection{Complex distributions}
A \emph{(complex) distribution} on $X$ is a sub-bundle\footnote{A
sub-bundle is an $\mathcal{O}_X$-submodule which is a direct
summand locally on $X$.} of ${\mathcal T}_X$.

A distribution $\mathcal{P}$ is called \emph{involutive} if it is closed under the Lie bracket, i.e. $[\mathcal{P},\mathcal{P}] \subseteq \mathcal{P}$.

For a distribution $\mathcal{P}$ on $X$ we denote by $\mathcal{P}^\perp \subseteq \Omega^1_X$ the annihilator of $\mathcal{P}$ (with respect to the canonical duality pairing).

A distribution $\mathcal{P}$ of rank $r$ on $X$ is called \emph{integrable} if, locally on $X$, there exist functions $f_1,\ldots , f_r\in\mathcal{O}_X$ such that $df_1,\ldots , df_r$ form a local frame for $\mathcal{P}^\perp$.

It is easy to see that an integrable distribution is involutive. The converse is true when $\mathcal{P}$ is \emph{real}, i.e. $\overline{\mathcal{P}} = \mathcal{P}$ (Frobenius) and when $\mathcal{P}$ is a \emph{complex structure}, i.e. $\overline{\mathcal{P}} \cap \mathcal{P} =0$ and $\overline{\mathcal{P}} \oplus \mathcal{P}
= \mathcal{T}_X$ (Newlander-Nirenberg). More generally, according to Theorem 1 of \cite{Rawnsley}, a sufficient condition for integrability of a complex distribution $\mathcal{P}$ is
\begin{equation}\label{sufficient condition}
\text{$\mathcal{P}\cap\overline{\mathcal{P}}$ is a sub-bundle and both $\mathcal{P}$ and $\mathcal{P} + \overline{\mathcal{P}}$ are involutive.}
\end{equation}

\subsection{The Hodge filtration}\label{subsection: the hodge filtration}
Suppose that $\mathcal{P}$ is an involutive distribution on $X$.

Let $F_\bullet\Omega^\bullet_X$ denote the filtration by the powers of the differential ideal generated by $\mathcal{P}^\perp$, i.e. $F_{-i}\Omega^j_X = \bigwedge^i{\mathcal
P}^\perp\wedge\Omega^{j-i}_X\subseteq\Omega^j_X$. Let $\dbar$ denote the differential in $Gr^F_\bullet\Omega^\bullet_X$. The wedge product of
differential forms induces a structure of a commutative DGA on
$(Gr^F_\bullet\Omega^\bullet_X,\dbar)$.

In particular, $Gr^F_0\mathcal{O}_X = \mathcal{O}_X$, $Gr^F_0\Omega^1_X = \Omega^1_X/\mathcal{P}^\perp$ and $\dbar\colon \mathcal{O}_X \to Gr^F_0\Omega^1_X$ is equal to the composition $\mathcal{O}_X \xrightarrow{d} \Omega^1_X \to \Omega^1_X/\mathcal{P}^\perp$. Let $\mathcal{O}_{X/\mathcal{P}} := \ker(\mathcal{O}_X \xrightarrow{\dbar} Gr^F_0\Omega^1_X)$. Equivalently, $\mathcal{O}_{X/\mathcal{P}} = (\mathcal{O}_X)^\mathcal{P} \subset \mathcal{O}_X$, the subsheaf of functions locally constant along $\mathcal{P}$. Note that $\dbar$ is $\mathcal{O}_{X/\mathcal{P}}$-linear.

Theorem 2 of \cite{Rawnsley} says that, if $\mathcal{P}$ satisfies the condition \eqref{sufficient condition}, the higher $\dbar$-cohomology of $\mathcal{O}_X$ vanishes, i.e.
\begin{equation}\label{d-bar cohomology of functions}
H^i(Gr^F_0\Omega^\bullet_X,\dbar) = \left\{
\begin{array}{ll}
\mathcal{O}_{X/\mathcal{P}} & \text{if $i=0$} \\
0 & \text{otherwise.}
\end{array}
\right.
\end{equation}

\medskip

\noindent
In what follows we will assume that the complex distribution $\mathcal{P}$ under consideration is integrable and satisfies \eqref{d-bar cohomology of functions}. The latter is implied by the condition \eqref{sufficient condition}.

\subsection{$\dbar$-operators} Suppose that $\mathcal{E}$ is a vector bundle on $X$, i.e. a locally free $\mathcal{O}_X$-module of finite rank. A \emph{connection along $\mathcal{P}$} on
$\mathcal{E}$ is, by definition, a map $\nabla^\mathcal{P}:
\mathcal{E}\to\Omega^1_X/\mathcal{P}^\perp\otimes_{{\mathcal
O}_X}\mathcal{E}$ which satisfies the Leibniz rule
$\nabla^\mathcal{P}(fe)=f\nabla^\mathcal{P}(e)+\overline\partial
f\cdot e$ for $e\in \mathcal{E}$ and $f \in \mathcal{O}_X$. Equivalently, a connection along $\mathcal{P}$ is an
$\mathcal{O}_X$-linear map $\nabla^\mathcal{P}_{(\bullet)}\colon \mathcal{P} \to \shEnd_\mathbb{C}(\mathcal{E})$ which satisfies the Leibniz rule $\nabla^\mathcal{P}_\xi(fe)=
f\nabla^\mathcal{P}_\xi(e)+\overline\partial f\cdot e$ for $e\in \mathcal{E}$ and $f \in \mathcal{O}_X$. In particular, $\nabla^\mathcal{P}_\xi$ is $\mathcal{O}_{X/\mathcal{P}}$-linear. The two avatars of a connection along $\mathcal{P}$
are related by $\nabla^\mathcal{P}_\xi(e)=\iota_\xi\nabla^\mathcal{P}(e)$.

A connection along $\mathcal{P}$ on $\mathcal{E}$ is called \emph{flat} if the corresponding map $\nabla^\mathcal{P}_{(\bullet)}\colon \mathcal{P} \to \shEnd_\mathbb{C}(\mathcal{E})$ is a morphism of Lie algebras. We will refer to a flat connection along $\mathcal{P}$ on
$\mathcal{E}$ as a $\dbar$-operator on $\mathcal{E}$.

A connection on $\mathcal{E}$ along $\mathcal{P}$ extends uniquely to a derivation $\dbar_\mathcal{E}$ of the graded $Gr^F_0\Omega^\bullet_X$-module $Gr^F_0\Omega^\bullet_X\otimes_{\mathcal{O}_X}\mathcal{E}$ which is a $\dbar$-operator if and only if $\dbar_{\mathcal E}^2=0$. The complex $(Gr^F_0\Omega^\bullet_X\otimes_{{\mathcal O}_X}\mathcal{E}, \dbar_\mathcal{E})$ is referred to as the (corresponding) $\dbar$-complex. Since $\dbar_\mathcal{E}$ is $\mathcal{O}_{X/\mathcal{P}}$-linear, the sheaves $H^i(Gr^F_0\Omega^\bullet_X\otimes_{{\mathcal O}_X}{\mathcal E},\dbar_\mathcal{E})$ are $\mathcal{O}_{X/\mathcal{P}}$-modules. The vanishing of higher $\dbar$-cohomology of $\mathcal{O}_X$ \eqref{d-bar cohomology of functions} generalizes easily to vector bundles.

\begin{lemma}\label{dbar lemma}
Suppose that $\mathcal{E}$ is a vector bundle and $\dbar_\mathcal{E}$ is a $\dbar$-operator on $\mathcal{E}$. Then, $H^i(Gr^F_0\Omega^\bullet_X \otimes_{\mathcal{O}_X}\mathcal{E}, \dbar_\mathcal
{E})=0$ for $i\neq 0$, i.e. the $\dbar$-complex is a resolution of $\ker(\dbar_\mathcal{E})$. Moreover, $\ker(\dbar_\mathcal{E})$ is locally free over $\mathcal{O}_{X/\mathcal{P}}$ of rank $\rk_{\mathcal{O}_X}\mathcal{E}$ and the map $\mathcal{O}_X \otimes_{\mathcal{O}_{X/\mathcal{P}}} \ker(\dbar_\mathcal{E}) \to \mathcal{E}$ (the $\mathcal{O}_X$-linear extension of the
inclusion $\ker(\dbar_\mathcal{E}) \hookrightarrow \mathcal{E}$) is an isomorphism.
\end{lemma}

\begin{remark}
Suppose that $\mathcal{F}$ is a locally free $\mathcal{O}_{X/\mathcal{P}}$-module of finite rank. Then, $\mathcal{O}_X\otimes_{\mathcal{O}_{X/\mathcal{P}}}\mathcal{F}$ is a locally free $\mathcal{O}_X$-module of rank $\rk_{\mathcal{O}_{X/\mathcal{P}}}\mathcal{F}$ and is endowed in a canonical way with a $\dbar$-operator, namely, $\dbar\otimes\id$.
The assignments $\mathcal{F}\mapsto(\mathcal{O}_X\otimes_{\mathcal{O}_{X/\mathcal{P}}}\mathcal{F},\dbar\otimes\id)$ and $(\mathcal{E},\dbar_\mathcal{E}) \mapsto \ker(\dbar_\mathcal{E})$ are mutually inverse equivalences of suitably defined categories.
\end{remark}

\subsection{Calculus}\label{subsection: calculus}
The adjoint action of $\mathcal{P}$ on $\mathcal{T}_X$ preserves $\mathcal{P}$, hence descends to an action on $\mathcal{T}_X/\mathcal{P}$. The latter action defines a connection along $\mathcal{P}$, i.e. a canonical $\dbar$-operator on $\mathcal{T}_X/\mathcal{P}$ which is easily seen to coincide with the one induced via the duality pairing between the latter and $\mathcal{P}^\perp$.\footnote{In the case of a real polarization this connection is known as the Bott connection.} Let $\mathcal{T}_{X/\mathcal{P}} := (\mathcal{T}_X/\mathcal{P})^\mathcal{P}$ (the subsheaf of $\mathcal{P}$ invariant section, equivalently, the kernel of the $\dbar$-operator on $\mathcal{T}_X/\mathcal{P}$. The Lie bracket on $\mathcal{T}_X$ (respectively, the action of $\mathcal{T}_X$ on $\mathcal{O}_X$) induces a Lie bracket on $\mathcal{T}_{X/\mathcal{P}}$ (respectively, an action of $\mathcal{T}_{X/\mathcal{P}}$ on $\mathcal{O}_{X/\mathcal{P}}$). The bracket and the action on $\mathcal{O}_{X/\mathcal{P}}$ endow $\mathcal{T}_{X/\mathcal{P}}$ with a structure of an $\mathcal{O}_{X/\mathcal{P}}$-Lie algebroid.

The action of $\mathcal{P}$ on $\Omega^1_X$ by Lie derivative restricts to a flat connection along $\mathcal{P}$, i.e. a canonical $\dbar$-operator on $\mathcal{P}^\perp$ and, therefore, on $\bigwedge^i\mathcal{P}^\perp$ for all $i$. It is easy to see that the multiplication map $Gr^F_0\Omega^\bullet\otimes\bigwedge^i\mathcal{P}^\perp \to Gr^F_{-i}\Omega^\bullet[i]$ is an isomorphism which identifies the $\dbar$-complex of $\bigwedge^i\mathcal{P}^\perp$ with $Gr^F_{-i}\Omega^\bullet[i]$. Let $\Omega^i_{X/\mathcal{P}}:= H^i(Gr^F_{-i}\Omega^\bullet_X,\dbar)$ (so that $\mathcal{O}_{X/\mathcal{P}}:=\Omega^0_{X/\mathcal{P}}$). Then,  $\Omega^i_{X/\mathcal{P}} \subset \bigwedge^i\mathcal{P}^\perp \subset \Omega^i_X$. The wedge product of differential forms
induces a structure of a graded-commutative algebra  on $\Omega^\bullet_{X/\mathcal {P}} := \oplus_i \Omega^i_{X/\mathcal{P}}[-i] = H^\bullet(Gr^F\Omega^\bullet_X,\dbar)$. The multiplication induces an isomorphism $\bigwedge^i_{\mathcal{O}_{X/\mathcal{P}}} \Omega^1_{X/\mathcal{P}} \to \Omega^i_{X/\mathcal{P}}$. The de Rham differential $d$ restricts to the map $d :
\Omega^i_{X/\mathcal{P}} \to \Omega^{i+1}_{X/\mathcal{P}}$ and the complex $\Omega^\bullet_{X/\mathcal{P}}:=(\Omega^\bullet_{X/\mathcal{P}},d)$ is a commutative DGA.

The Hodge filtration $F_\bullet\Omega^\bullet_{X/\mathcal{P}}$ is defined by
\[
F_i\Omega^\bullet_{X/\mathcal{P}} = \oplus_{j\geq -i} \Omega^j_{X/\mathcal{P}} ,
\]
so that the inclusion $\Omega^\bullet_{X/\mathcal{P}} \hookrightarrow \Omega^\bullet_X$ is filtered with respect to the Hodge filtration. It follows from Lemma \ref{dbar lemma} that it is, in fact, a filtered quasi-isomorphism.

The duality pairing $\mathcal{T}_X/\mathcal{P}\otimes\mathcal{P}^\perp \to \mathcal{O}_X$ restricts to a non-degenerate pairing $\mathcal{T}_{X/\mathcal{P}} \otimes_{\mathcal{O}_{X/\mathcal{P}}} \Omega^1_{X/\mathcal{P}} \to \mathcal{O}_{X/\mathcal{P}}$.
The action of $\mathcal{T}_X/\mathcal{P}$ on $\mathcal{O}_{X/\mathcal{P}}$ the pairing and the de Rham differential are related by the usual formula $\xi(f)=\iota_\xi df$, for $\xi \in \mathcal{T}_{X/\mathcal{P}}$ and $f\in\mathcal{O}_{X/\mathcal{P}}$.

\section{Algebroid stacks}\label{section: algebroid stacks}
In what follows we will be considering gerbes with abelian lien. Below we briefly recall some relevant notions and constructions with the purpose of establishing notations. We refer the reader to \cite{D} and \cite{M} for detailed treatment of Picard stacks and gerbes respectively.

Suppose that $X$ is a topological space.

\subsection{Picard stacks}
We recall the definitions from \emph{1.4.Champs de Picard strictement commutatifs} of \cite{D}. 

A (strictly commutative) \emph{Picard groupoid} $\mathcal{P}$ is a non-empty groupoid equipped with a functor $+ \colon \mathcal{P}\times\mathcal{P}\to\mathcal{P}$ and functorial isomorphisms
\begin{itemize}
\item $\sigma_{x,y,z} \colon (x+y)+z \to x+(y+z)$
\item $\tau_{x,y} \colon x+y \to y+x$
\end{itemize}
rendering $+$ associative and strictly commutative, and such that for each object $x\in\mathcal{P}$ the functor $y \mapsto x+y$ is an equivalence.

A \emph{Picard stack} on $X$ is a stack in groupoids $\mathcal{P}$ equipped with a functor
$+ \colon \mathcal{P}\times\mathcal{P}\to\mathcal{P}$ and functorial isomorphisms $\sigma$ and $\tau$ as above, which, for each open subset $U \subseteq X$, endow the category $\mathcal{P}(U)$ a structure of a Picard groupoid.

\subsection{Torsors}
Suppose that $A$ is a sheaf of abelian groups on $X$. The stack of $A$-torsors will be denoted by $A[1]$; it is a gerbe since all $A$-torsors are locally trivial.

Suppose that $\phi \colon A \to B$ is a morphism of sheaves of abelian groups. The assignment $A[1] \ni T \mapsto \phi(T):= T\times_A B \in B[1]$ extends to a morphism $\phi \colon A[1] \to B[1]$ of stacks. There is a canonical map of sheaves of torsors $\phi = \phi_T \colon T \to \phi(T)$ compatible with the map $\phi$ of abelian groups and respective actions.

Suppose that $A$ and $B$ are sheaves of abelian groups. The assignment $A[1]\times B[1] \ni (S,T) \mapsto S\times T \in (A\times B)[1]$ extends to a morphism of stacks $\times \colon A[1]\times B[1] \to (A\times B)[1]$.

Suppose that $A$ is a sheaf of abelian groups with the group structure $+ \colon A\times A \to A$. The latter is a morphism of sheaves of groups since $A$ is abelian. The assignment $A[1]\times A[1] \ni (S,T) \mapsto S+T := +(S\times T)$ defines a structure of a Picard stack on $A[1]$. If $\phi \colon A \to B$ is a morphism of sheaves of abelian groups the corresponding morphism $\phi \colon A[1] \to B[1]$ is a morphism of Picard stacks.

As a consequence, the set $\pi_0A[1](X)$ of isomorphism classes of $A$-torsors is endowed with a canonical structure of an abelian group. There is a canonical isomorphism of groups $\pi_0A[1](X) \cong H^1(X;A)$.

\subsection{Gerbes}
An $A$-gerbe is a stack in groupoids which is a twisted form of (i.e. locally equivalent to) $A[1]$. The (2-)stack of $A$-gerbes will be denoted $A[2]$. Since all $A$-gerbes are locally equivalent the (2-)stack $A[2]$ is a (2-)gerbe. Every $A$-gerbe $\mathcal{S}$ (is equivalent to one which) admits a canonical action of the Picard stack $A[1]$ by autoequivalences denoted $+ \colon A[1]\times \mathcal{S} \to \mathcal{S}$, $(T, L) \mapsto T+L$ endowing $\mathcal{S}$ with a structure of a (2-)torsor under $A[1]$. We shall not make distinction between $A$-gerbes and  (2-)torsors under $A[1]$ and use the notation $A[2]$ for both.

Suppose that $\phi \colon A \to B$ is a morphism of sheaves of abelian groups and $\mathcal{S}$ is an $A$-gerbe. In particular, for any two (locally defined) objects $s_1, s_2 \in \mathcal{S}$ the sheaf $\shHom_\mathcal{S}(s_1, s_2)$ is an $A$-torsor. The stack $\phi\mathcal{S}$ is defined as the stack associated to the prestack with the same objects as $\mathcal{S}$ and $\shHom_{\phi\mathcal{S}}(s_1, s_2) := \phi(\shHom_\mathcal{S}(s_1, s_2))$. Then, $\phi\mathcal{S}$ is a $B$-gerbe and the assignment $\mathcal{S} \mapsto \phi\mathcal{S}$ extends to a morphism $\phi \colon A[2] \to B[2]$. There is a canonical morphism of stacks $\phi = \phi_\mathcal{S} \colon \mathcal{S} \to \phi\mathcal{S}$ which induces the map $\phi \colon A \to B$ on groups of automorphisms.

The Picard structure on $A[1]$ gives rise to one on $A[2]$ defined in analogous fashion. As a consequence, the set $\pi_0A[2](X)$ of equivalence classes of $A[1]$-torsors is endowed with a canonical structure of an abelian group. There is a canoncial isomorphism of groups $\pi_0A[2](X) \cong H^2(X;A)$.

Let $A^0 \xrightarrow{d} A^1$ be a complex of sheaves of abelian groups on $X$ concentrated in degrees zero and one. Recall that a $(A^0 \xrightarrow{d} A^1)$-torsor is a pair $(T, \tau)$, where $T$ is a $A^0$-torsor and $\tau$ is a trivialization (i.e. a section) of the $A^1$-torsor $d(T) = T\times_{A^0} A^1$. A morphism of $(A^0 \xrightarrow{d} A^1)$-torsors $\phi \colon (S,\sigma) \to (T,\tau)$ is a morphism of $A^0$-torsors $\phi \colon S \to T$ such that the induced morphism of $A^1$-torsors $d(\phi) \colon d(S) \to d(T)$ which commutes with respective trivializations, i.e. $d(\phi)(\sigma) = \tau$

The monoidal structure on the category of $(A^0 \xrightarrow{d} A^1)$-torsors is defined as follows. Suppose that $(S,\sigma)$ and $(T,\tau)$ are $(A^0 \xrightarrow{d} A^1)$-torsors. The sum $(S,\sigma)+(T,\tau)$ is represented by $(S+T, \sigma + \tau)$, where $\sigma + \tau$ is the trivialization of $d(S)+d(T) = d(S+T)$ induced by $\sigma$ and $\tau$.

Locally defined $(A^0 \xrightarrow{d} A^1)$-torsors form a Picard stack on $X$ which we will denote by $(A^0 \xrightarrow{d} A^1)[1]$. By a result of P.~Deligne (\cite{D}, Proposition 1.4.15) all Picard stacks arise in this way. The group (under the operation induced by the monoidal structure) of isomorphism classes of $(A^0 \xrightarrow{d} A^1)$-torsors on $X$, i.e. $\pi_0(A^0 \xrightarrow{d} A^1)[1](X)$ is canonically isomorphic to $H^1(X;A^0 \xrightarrow{d} A^1)$.

\subsection{2-torsors}
A $(A^0 \xrightarrow{d} A^1)$-gerbe is equivalent to the data $(\mathcal{S}, \sigma)$, where $\mathcal{S}$ is an $A^0$-gerbe and $\tau$ is a trivialization of the $A^1$-gerbe $d\mathcal{S}$, i.e. an equivalence $\tau \colon d\mathcal{S} \to A^1[1]$. The composition $\mathcal{S} \xrightarrow{d} d\mathcal{S} \xrightarrow{\tau} A^1[1]$ is a functorial assignment of an $A^1$-torsor $\tau(s)$ to a (locally defined) object $s\in\mathcal{S}$. The 2-stack of $(A^0 \xrightarrow{d} A^1)$-gerbes will be denoted $(A^0 \xrightarrow{d} A^1)[2]$. The Picard structure on $(A^0 \xrightarrow{d} A^1)[1]$ gives rise to one on $(A^0 \xrightarrow{d} A^1)[2]$ defined in an analogous fashion. As a consequence, the set $\pi_0(A^0 \xrightarrow{d} A^1)[2](X)$ of equivalence classes of $(A^0 \xrightarrow{d} A^1)$-gerbes is endowed with a canonical structure of an abelian group. There is a canonical isomorphism of groups $\pi_0(A^0 \xrightarrow{d} A^1)[2](X) \cong H^2(X;A^0 \xrightarrow{d} A^1)$.

\subsection{Algebroids}
Recall that a $k$-algebroid, $k$ a commutative ring with unit, is a stack in $k$-linear categories $\mathcal{C}$ such that the substack of isomorphisms $i\mathcal{C}$ (which is a stack in groupoids) is a gerbe.

Suppose that $\mathcal{A}$ is a sheaf of $k$-algebras on $X$. For $U$ open in $X$ let $\mathcal{A}^+(U)$ denote the category with one object whose endomorphism algebra is $\mathcal{A}^\op$. The assignment $U\mapsto \mathcal{A}^+(U)$ defines a prestack on $X$, which we denote $\mathcal{A}^+$. Let $\widetilde{\mathcal{A}^+}$ denote the associated stack. Note that $\widetilde{\mathcal{A}^+}$ is equivalent to the stack of $\mathcal{A}$-modules locally isomorphic to $\mathcal{A}$. The category $\widetilde{\mathcal{A}^+}(X)$ is non-empty.

Conversely, let $\mathcal{C}$ be a $k$-algebroid such that the category $\mathcal{C}(X)$ is non-empty. Let $L\in\mathcal{C}(X)$, and let $\mathcal{A} := \shEnd_\mathcal{C}(L)^\op$. The morphism $\mathcal{A}^+ \to \mathcal{C}$ which sends the unique object to $L$ induces and equivalence $\widetilde{\mathcal{A}^+} \to \mathcal{C}$.

For a sheaf of $k$-algebras $\mathcal{A}$ on $X$ a \emph{twisted form of $\mathcal{A}$} is a $k$-algebroid locally $k$-linearly equivalent to $\widetilde{\mathcal{A}^+}$. 

Let $X$ be a $C^\infty$ manifold equipped with an integrable complex distribution $\mathcal{P}$. Twisted forms of the sheaf of $\mathbb{C}$-algebras $\mathcal{O}_{X/\mathcal{P}}$ are in one-to-one correspondence with $\mathcal{O}_{X/\mathcal{P}}^\times$-gerbes via $\mathcal{S} \mapsto i\mathcal{S}$. Equivalence classes of twisted forms of $\mathcal{O}_{X/\mathcal{P}}$ form an group canonically isomorphic to $H^2(X;\mathcal{O}_{X/\mathcal{P}}^\times)$.

\section{Lie algebroids}\label{section: lie algebroids}
In this section we review basic definitions and facts concerning Lie algebroids, modules over Lie algebroids and cohomological classification of invertible modules. General references for this material are \cite{Mac05}

Throughout the section
$X$ is a $C^\infty$ manifold equipped with an integrable complex distribution $\mathcal{P}$, see Section \ref{section: distributions} for definitions and notations.

\subsection{$\mathcal{O}_{X/\mathcal{P}}$-Lie algebroids}
An $\mathcal{O}_{X/\mathcal{P}}$-Lie algebroids or, simply, a Lie algebroid is the datum of
\begin{itemize}
\item an $\mathcal{O}_{X/\mathcal{P}}$-module $\mathcal{B}$,
\item an $\mathcal{O}_{X/\mathcal{P}}$-linear map $\sigma\colon \mathcal{B} \to \mathcal{T}_{X/\mathcal{P}}$ (the \emph{anchor}),
\item a structure of a $\mathbb{C}$-Lie algebra on $\mathcal{B}$ given by $[\ ,\ ]\colon \mathcal{B}\otimes_\mathbb{C}\mathcal{B} \to \mathcal{B}$
\end{itemize}
such that the Leibniz rule
\[
[b_1,fb_2] = \sigma(b_1)(f)b_2 + f[b_1,b_2],
\]
$b_1, b_2\in\mathcal{B}$, $f\in\mathcal{O}_{X/\mathcal{P}}$, holds. As a consequence, the anchor map is a morphism of Lie algebras.

A \emph{morphism of Lie algebroids} is an $\mathcal{O}_{X/\mathcal{P}}$-linear map $\phi\colon \mathcal{B}_1 \to \mathcal{B}_2$ of Lie algebras which is commutes with respective anchors.

The sheaf $\mathcal{T}_{X/\mathcal{P}}$ is a Lie algebroid under the Lie bracket of vector fields with the anchor the identity map. It is immediate from the definitions that $\mathcal{T}_{X/\mathcal{P}}$ is the terminal object in the category of Lie algebroids.

Further examples of Lie algebroids are given below.

\subsubsection{Atiyah algebras}\label{subsubsection: Atiyah algebras}
For a locally free $\mathcal{O}_{X/\mathcal{P}}$-module of finite rank $\mathcal{E}$ the \emph{Atiyah algebra} of $\mathcal{E}$, denoted $\mathcal{A}_\mathcal{E}$ is defined by the pull-back diagram
\begin{equation*}
\begin{CD}
\mathcal{A}_\mathcal{E} @>{\sigma}>> \mathcal{T}_{X/\mathcal{P}} \\
@VVV @VV{1\otimes\id}V \\
\operatorname{Diff}^{\leq 1}(\mathcal{E}, \mathcal{E}) @>{\sigma}>> \shEnd_{\mathcal{O}_{X/\mathcal{P}}}(\mathcal{E})\otimes\mathcal{T}_{X/\mathcal{P}}
\end{CD}
\end{equation*}
where the bottom horizontal arrow is the order one principal symbol map.
It is a Lie algebroid under (the restriction of) the commutator bracket, the anchor given by the top horizontal map (the restriction of the principal symbol map). There is an exact sequence
\begin{equation*}
0 \to \shEnd_{\mathcal{O}_{X/\mathcal{P}}}(\mathcal{E}) \to \mathcal{A}_\mathcal{E} \xrightarrow{\sigma} \mathcal{T}_{X/\mathcal{P}} \to 0 \ .
\end{equation*}

\subsubsection{Poisson structures}\label{subsubsection: Poisson structures}
Let $\poisson\in\Gamma(X;\bigwedge^2\mathcal{T}_{X/\mathcal{P}})$ be a bi-vector; we denote by $\widetilde{\poisson}\colon \Omega^1_{X/\mathcal{P}} \to \mathcal{T}_{X/\mathcal{P}}$ the adjoint map $\alpha\mapsto\poisson(\alpha,\cdot)$. The bi-vector $\poisson$ is \emph{Poisson} if the operation $\{\ ,\ \}\colon \mathcal{O}_{X/\mathcal{P}}\otimes_\mathbb{C}\mathcal{O}_{X/\mathcal{P}} \to \mathcal{O}_{X/\mathcal{P}}$ defined by $\{f,g\} = \poisson(df,dg)$ satisfies the Jacobi identity, i.e. endows $\mathcal{O}_{X/\mathcal{P}}$ with a structure of a $\mathbb{C}$-Lie algebra.

Let $[\ ,\ ]_\poisson\colon \Omega^1_{X/\mathcal{P}}\otimes_\mathbb{C}\Omega^1_{X/\mathcal{P}} \to \Omega^1_{X/\mathcal{P}}$ denote the map given by
\begin{equation}\label{bracket on forms}
[\alpha,\beta]_\poisson = L_{\widetilde{\poisson}(\alpha)}\beta -  L_{\widetilde{\poisson}(\beta)}\alpha - d\poisson(\alpha,\beta) .
\end{equation}
The bi-vector $\poisson$ is Poisson if and only if the bracket \eqref{bracket on forms} on $\Omega^1_{X/\mathcal{P}}$ satisfies the Jacobi identity; the map $\widetilde{\poisson}$ is compatible with respective brackets.

Thus, a Poisson bi-vector $\poisson$ gives rise to a structure, denoted $\Pi$ in what follows, of a Lie algebroid on $\Omega^1_{X/\mathcal{P}}$ with bracket \eqref{bracket on forms} and the anchor $\widetilde{\poisson}$. 

\subsubsection{The de Rham complex}\label{subsubsection: the de Rham complex}
We shall assume for simplicity that the Lie algebroid $\mathcal{B}$ in question is locally free of finite rank over $\mathcal{O}_{X/\mathcal{P}}$. Let $\Omega^0_\mathcal{B} = \mathcal{O}_{X/\mathcal{P}}$, $\Omega^1_\mathcal{B} = \dual{\mathcal{B}}$, where, for an $\mathcal{O}_{X/\mathcal{P}}$-module $\mathcal{E}$ we put $\dual{\mathcal{E}} := \shHom_{\mathcal{O}_{X/\mathcal{P}}}(\mathcal{E}, \mathcal{O}_{X/\mathcal{P}})$.

Let $\Omega^i_\mathcal{B} = \bigwedge^i\Omega^1_\mathcal{B}$. Then, $\Omega^\bullet_\mathcal{B} := \bigoplus_i \Omega^i_\mathcal{B}[-i]$ is a graded commutative algebra. The composition
\begin{equation*}
\mathcal{O}_{X/\mathcal{P}} \xrightarrow{d} \Omega^1_{X/\mathcal{P}} \xrightarrow{\dual{\sigma}} \dual{B} = \Omega^1_\mathcal{B}
\end{equation*}
extends uniquely to a square-zero derivation $d_\mathcal{B}$ of degree one of the algebra $\Omega^\bullet_\mathcal{B}$. Let $\Omega^{i, cl}_\mathcal{B} := \ker(\Omega^i_\mathcal{B} \xrightarrow{d_\mathcal{B}} \Omega^{i+1}_\mathcal{B})$. In the case $\mathcal{B} = 
\Pi$ of example \ref{subsubsection: Poisson structures} $\Omega^i_\Pi = \bigwedge^i\mathcal{T}_{X/\mathcal{P}}$ and the de Rham differential is given by $d_
\Pi = [\poisson, \cdot ]$, i.e the Schouten bracket with the bi-vector $\pi$.

\subsection{Modules over Lie algebroids}
Suppose that $\mathcal{B}$ is a Lie algebroid.

\subsubsection{$\mathcal{B}$-modules}
Suppose that $\mathcal{M}$ is a $\mathcal{O}_{X/\mathcal{P}}$-module. A $\mathcal{B}$-module structure on $\mathcal{M}$ is given by an $\mathcal{O}_{X/\mathcal{P}}$-linear map $\mathcal{B} \to \shEnd_\mathbb{C}(\mathcal{M})$, $b \mapsto (m\mapsto b.m)$, $b\in\mathcal{B}$, $m\in\mathcal{M}$, which is a map of Lie algebras (with respect to the commutator bracket of endomorphisms) and satisfies the Leibniz rule $b.fm = \sigma(b)(f)m + f\cdot b.m$.

Suppose that $\mathcal{E}$ is a locally free $\mathcal{O}_{X/\mathcal{P}}$-module of finite rank. It transpires from the above definition that a $\mathcal{B}$-module structure on $\mathcal{E}$ amounts to a morphism of algebroids $\mathcal{B} \to \mathcal{A}_\mathcal{E}$. Equivalently, it is a morphism of Lie algebroids $\mathcal{B} \to \mathcal{B}\times_{\mathcal{T}_{X/\mathcal{P}}} \mathcal{A}_\mathcal{E}$ splitting the exact sequence
\begin{equation*}
0 \to \shEnd_{\mathcal{O}_{X/\mathcal{P}}}(\mathcal{E}) \to \mathcal{B}\times_{\mathcal{T}_{X/\mathcal{P}}} \mathcal{A}_\mathcal{E} \xrightarrow{\pr_\mathcal{B}} \mathcal{B} \to 0 .
\end{equation*}

A morphism of $\mathcal{B}$-modules is an $\mathcal{O}_{X/\mathcal{P}}$-linear map which commutes with respective actions of $\mathcal{B}$. $\mathcal{B}$-modules form an Abelian category.

The structure sheaf $\mathcal{O}_{X/\mathcal{P}}$ is a $\mathcal{T}_{X/\mathcal{P}}$-module hence a $\mathcal{B}$-module in a canonical way for any Lie algebroid $\mathcal{B}$.

\subsubsection{Tensor product of $\mathcal{B}$-modules}
The category of $\mathcal{B}$-modules is endowed with a monoidal structure: for any two modules $\mathcal{M}$ and $\mathcal{N}$ their tensor product $\mathcal{M}\otimes_{\mathcal{O}_{X/\mathcal{P}}}\mathcal{N}$ is equipped with the canonical $\mathcal{B}$-module structure given by the Leibniz rule. Namely, for $m\in\mathcal{M}$, $n\in\mathcal{M}$, $b\in\mathcal{B}$, $b.(m\otimes n) = (b.m)\otimes n + m\otimes(b.n)$.

The sheaf $\shHom_{\mathcal{O}_{X/\mathcal{P}}}(\mathcal{M}, \mathcal{N})$ is equipped with the canonical structure of a $\mathcal{B}$-module by $(b.\phi)(m) = b.(\phi(m)) - \phi(b.m)$.

\subsubsection{Invertible $\mathcal{B}$-modules}\label{subsubsection: invertible B-modules}
For the purposes of this note an invertible $\mathcal{B}$-module is a line bundle (i.e. a locally free $\mathcal{O}_{X/\mathcal{P}}$-module of rank one) equipped with a structure of a $\mathcal{B}$-module. Invertible $\mathcal{B}$-modules form a Picard category, denoted $\Pic_\mathcal{B}(X)$, under the tensor product over $\mathcal{O}_{X/\mathcal{P}}$. Let $\shPic_\mathcal{B}$ denote the Picard stack of invertible $\mathcal{B}$-modules, i.e. $\underline{\Pic}_\mathcal{B}(U) = \Pic_\mathcal{B}(U)$ for $U$ open in $X$.

Suppose that $\mathcal{L}$ is a line bundle; we shall denote by $\mathcal{L}^\times$ the corresponding $\mathcal{O}_{X/\mathcal{P}}^\times$-torsor (the subsheaf of nowhere vanishing sections). Recall that a $\mathcal{B}$-module structure on $\mathcal{L}$ is a morphism of Lie algebroids $\nabla\colon \mathcal{B} \to \mathcal{B}\times_{\mathcal{T}_{X/\mathcal{P}}} \mathcal{A}_\mathcal{L}$ splitting the short exact sequence
\begin{equation}\label{O ext B struct on L}
0 \to \mathcal{O}_{X/\mathcal{P}} \xrightarrow{i} \mathcal{B}\times_{\mathcal{T}_{X/\mathcal{P}}} \mathcal{A}_\mathcal{L} \xrightarrow{\pr_\mathcal{B}} \mathcal{B} \to 0 .
\end{equation}
For any two such, $\nabla_1, \nabla_2 \colon \mathcal{B} \to \mathcal{B}\times_{\mathcal{T}_{X/\mathcal{P}}} \mathcal{A}_\mathcal{L}$, their difference $\nabla_2 - \nabla_1$ satisfies $\pr_\mathcal{B}\circ(\nabla_2 - \nabla_1) = 0$ and therefore factors through (the inclusion of) $\mathcal{O}_{X/\mathcal{P}}$. Moreover, the resulting section $\nabla_2 - \nabla_1 \in \Hom_{\mathcal{O}_{X/\mathcal{P}}}(\mathcal{B},\mathcal{O}_{X/\mathcal{P}}) = \Omega^1_\mathcal{B}$ satisfies $d_\mathcal{B}(\nabla_2 - \nabla_1) = 0$, where $d_\mathcal{B}$ is the de Rham differential introduced in \ref{subsubsection: the de Rham complex}.

The sheaf of (locally defined) splittings of \eqref{O ext B struct on L} by Lie algebroid morphisms is a $\Omega^{1,cl}_\mathcal{B}$-torsor and will be denoted $\CONN^\flat_\mathcal{B}(\mathcal{L})$. A structure of a $\mathcal{B}$-module on $\mathcal{L}$ is a trivialization of $\CONN^\flat_\mathcal{B}(\mathcal{L})$.

The morphism of sheaves of groups
\[
d_\mathcal{B}\log \colon \mathcal{O}_{X/\mathcal{P}}^\times \to \Omega^{1,cl}_\mathcal{B}
\]
is defined by $f \mapsto \dfrac{d_\mathcal{B}f}{f}$.
There is a canonical morphism of sheaves
\begin{equation}\label{dlog L}
d_\mathcal{B}\log\colon \mathcal{L}^\times \to \CONN^\flat_\mathcal{B}(\mathcal{L})
\end{equation}
defined as follows. Let $s \in \mathcal{L}^\times$ be a (locally defined) section. The section $s$ gives rise to a (locally defined) isomorphism $\mathcal{O}_{X/\mathcal{P}} \xrightarrow{\cong} \mathcal{L}$. Via this isomorphism the canonical $\mathcal{B}$-module structure on $\mathcal{O}_{X/\mathcal{P}}$ gives rise to a $\mathcal{B}$-module structure on $\mathcal{L}$ which corresponds to a (locally defined) section denoted $d_\mathcal{B}\log(s) \in \CONN^\flat_\mathcal{B}(\mathcal{L})$. We leave it to the reader to check that the morphism \eqref{dlog L} is compatible with the map of sheaves of groups $d_\mathcal{B}\log$ and respective actions. Hence, \eqref{dlog L} induces a canonical isomorphism of $\Omega^{1,cl}_\mathcal{B}$-torsors $\CONN^\flat_\mathcal{B}(\mathcal{L})\cong d_\mathcal{B}\log(\mathcal{L}^\times) := \mathcal{L}^\times\times_{\mathcal{O}_{X/\mathcal{P}}^\times} \Omega^{1,cl}_\mathcal{B}$. In particular, a structure of a $\mathcal{B}$-module on $\mathcal{L}$ amounts to a trivialization of $d_\mathcal{B}\log(\mathcal{L}^\times)$.

\subsubsection{Classification of invertible $\mathcal{B}$-modules}
Assigning to an invertible $\mathcal{B}$-module $\mathcal{L}$ the pair $(\mathcal{L}^\times, \nabla)$, where $\nabla$ is the trivialization of $d_\mathcal{B}\log(\mathcal{L}^\times)$ corresponding to the $\mathcal{B}$-module structure we obtain a morphism of Picard stacks
\begin{equation}\label{Pic B to tors}
\shPic_\mathcal{B} \to (\mathcal{O}_{X/\mathcal{P}}^\times \xrightarrow{d_\mathcal{B}\log} \Omega_\mathcal{B}^{1,cl})[1]
\end{equation}

\begin{lemma}\label{lemma: pic algd}
The morphism \eqref{Pic B to tors} is an equivalence.
\end{lemma}
\begin{proof}
A quasi-inverse to \eqref{Pic B to tors} is given by the following construction.
A $\mathcal{O}_{X/\mathcal{P}}^\times$-torsor $T$ determines the line bundle $\mathcal{L} := T\times_{\mathcal{O}_{X/\mathcal{P}}^\times}\mathcal{O}_{X/\mathcal{P}}$ such that there is a canonical isomorphism $\mathcal{L}^\times \cong T$, hence $d_\mathcal{B}\log(T) \cong d_\mathcal{B}\log(\mathcal{L}^\times)\cong \CONN^\flat_\mathcal{B}(\mathcal{L})$. Thus, a trivialization of $d_\mathcal{B}\log(T)$ gives rise to a $\mathcal{B}$-module structure on $\mathcal{L}$.
\end{proof}

\begin{cor}
$\pi_0\Pic_\mathcal{B}(X) \cong H^1(X;\mathcal{O}_{X/\mathcal{P}}^\times \xrightarrow{d_\mathcal{B}\log} \Omega_\mathcal{B}^{1,cl})$.
\end{cor}

The above result reduces to the well-known classification of line bundles equipped with flat connections in the case $\mathcal{B} = \mathcal{T}_{X/\mathcal{P}}$ (see, for example, \cite{JLB}). In the case $\mathcal{B} = \Pi$ (see \ref{subsubsection: Poisson structures}) Lemma \ref{lemma: pic algd} is a particular case of Proposition 2.8 of \cite{B01}.

\subsection{$\mathcal{O}$-extensions}
Suppose that $\mathcal{B}$ is a $\mathcal{O}_{X/\mathcal{P}}$-Lie algebroid. An \emph{$\mathcal{O}$-extension of $\mathcal{B}$} is a triple $(\widetilde{\mathcal{B}}, \mathfrak{c}, \sigma)$ which consists of
\begin{itemize}
\item a Lie algebroid $\widetilde{\mathcal{B}}$,
\item a central element $\mathfrak{c}$ of the Lie algebra of sections $\Gamma(X;\widetilde{\mathcal{B}})$,
\item a morphism of Lie algebroids
$\sigma : \widetilde{\mathcal{B}} \to \mathcal{B}$
\end{itemize}
such that these data give rise to the associated short exact sequence
\begin{equation}\label{O-ext}
0 \to \mathcal{O}_{X/\mathcal{P}} \xrightarrow{f\mapsto f\cdot\mathfrak{c}} \widetilde{\mathcal{B}} \xrightarrow{\sigma} \mathcal{B} \to 0
\end{equation}
Locally defined $\mathcal{O}$-extensions of $\mathcal{B}$ form a Picard stack under the operation of Baer sum of extensions  which we denote $\mathcal{O}\EXT(\mathcal{B})$.

Since $\mathfrak{c}$ is central, it follows from the Leibniz rule that, for $f\in \mathcal{O}_{X/\mathcal{P}}$, $b\in \widetilde{\mathcal{B}}$, $[b,f\cdot\mathfrak{c}] = \overline{b}(f)\cdot\mathfrak{c}$, where $\overline{b} \in \mathcal{T}_{X/\mathcal{P}}$ denotes the image of $b$ under the anchor map. That is to say, the (adjoint) action of $\widetilde{\mathcal{B}}$ on $\mathcal{O}_{X/\mathcal{P}}\cdot\mathfrak{c} \cong \mathcal{O}_{X/\mathcal{P}}$ factors through $\mathcal{T}_{X/\mathcal{P}}$ and the latter action coincides with the Lie derivative action of vector fields on functions.

Suppose that $\mathcal{B}$ is locally free of finite rank over $\mathcal{O}_{X/\mathcal{P}}$. Then there is a canonical equivalence of Picard stacks
\begin{equation}\label{O-ext to tors}
\mathcal{O}\EXT(\mathcal{B})\xrightarrow{\cong} (\Omega^1_\mathcal{B} \to \Omega^{2,cl}_\mathcal{B})[1] .
\end{equation}
The functor \eqref{O-ext to tors} associates to $(\widetilde{\mathcal{B}}, \mathfrak{c}, \sigma)$ the $\Omega^1_\mathcal{B}$-torsor $\CONN_\mathcal{B}(\widetilde{\mathcal{B}})$ of (locally defined) splittings of $\sigma \colon \widetilde{\mathcal{B}} \to \mathcal{B}$. The ``curvature" map $c \colon \CONN_\mathcal{B}(\widetilde{\mathcal{B}}) \to \Omega^{2,cl}_\mathcal{B}$, $\CONN_\mathcal{B}(\widetilde{\mathcal{B}}) \ni \nabla \mapsto c(\nabla) \in \Omega^{2,cl}_\mathcal{B} \subset \Hom(\bigwedge^2 \mathcal{B}, \mathcal{O}_{X/\mathcal{P}})$ is determined by $c(\nabla)(b_1,b_2)\cdot\mathfrak{c} = [\nabla(b_1), \nabla(b_2)] - \nabla([b_1,b_2])$. We leave it to the reader to check that the standard calculation shows that the above formula defines a closed 2-form.

\section{Connective structures}\label{section: connective structures}
Throughout this section $X$ is a $C^\infty$ manifold equipped with an integrable complex distribution $\mathcal{P}$ (see Section \ref{section: distributions} for definitions and notations), $\mathcal{B}$ is a $\mathcal{O}_{X/\mathcal{P}}$-Lie algebroid (Section \ref{section: lie algebroids}). This section is devoted to generalities on connective structures on $\mathcal{O}_{X/\mathcal{P}}^\times$-gerbes in terms of the formalism of Section \ref{section: algebroid stacks}. The exposition follows closely the terminology of \cite{JLB}, Chapter V.

\subsection{$\mathcal{B}$-connective structures}
Let $\mathcal{S}$ be a $\mathcal{O}_{X/\mathcal{P}}^\times$-gerbe. A \emph{$\mathcal{B}$-connective structure} on $\mathcal{S}$ is a trivialization of the $\Omega^1_\mathcal{B}$-gerbe $d_\mathcal{B}\log\mathcal{S}$, i.e. a morphism, thus an equivalence, of $\Omega^1_\mathcal{B}$-gerbes $\Conn\colon d_\mathcal{B}\log\mathcal{S} \to \Omega^1_\mathcal{B}[1]$. The composition $\mathcal{S} \to d_\mathcal{B}\log\mathcal{S} \xrightarrow{\Conn} \Omega^1_\mathcal{B}[1]$ is a functorial assignment of a $\Omega^1_\mathcal{B}$-torsor $\Conn(L)$ to a (locally defined) object $L\in\mathcal{S}$ which induces the map $d_\mathcal{B}\log\colon \mathcal{O}_{X/\mathcal{P}}^\times \to \Omega^1_\mathcal{B}$ on respective sheaves of groups of automorphisms. For locally defined objects $L_1, L_2\in\mathcal{S}$ the connective structure induces the isomorphism $\Conn\colon \CONN_\mathcal{B}(\shHom_\mathcal{S}(L_1,L_2)) = \shHom_{d_\mathcal{B}\log\mathcal{S}}(d_\mathcal{B}\log(L_1),d_\mathcal{B}\log(L_2)) \to \shHom_{\Omega^1_\mathcal{B}[1]}(\Conn(L_1),\Conn(L_2)) \cong \Conn(L_2)-\Conn(L_1)$.

It transpires from the above description that locally defined $\mathcal{O}_{X/\mathcal{P}}^\times$-gerbes with $\mathcal{B}$-connective structure, i.e. pairs $(\mathcal{S},\Conn)$ as above are objects of the Picard 2-stack $(\mathcal{O}_{X/\mathcal{P}}^\times \xrightarrow{d_\mathcal{B}\log} \Omega^1_\mathcal{B})[2]$.
As a consequence, the group of equivalence classes of $\mathcal{O}_{X/\mathcal{P}}^\times$-gerbes with $\mathcal{B}$-connective structure is canonically isomorphic to $H^2(X;\mathcal{O}_{X/\mathcal{P}}^\times \xrightarrow{d_\mathcal{B}\log} \Omega^1_\mathcal{B})$.

Since the map $d_\mathcal{B}\log\colon \mathcal{O}_{X/\mathcal{P}}^\times \to \Omega^1_\mathcal{B}$ factors through the inclusion $\Omega^{1,cl}_\mathcal{B} \hookrightarrow \Omega^1_\mathcal{B}$ the $\Omega^1_\mathcal{B}$-gerbe $d_\mathcal{B}\log\mathcal{S}$ is induced from the $\Omega^{1,cl}_\mathcal{B}$-gerbe which will be denoted $d_\mathcal{B}\log\mathcal{S}^\flat$. A connective structure $\Conn\colon d_\mathcal{B}\log\mathcal{S} \to \Omega^1_\mathcal{B}[1]$ is called \emph{flat} if it is induced from the trivialization of $d_\mathcal{B}\log\mathcal{S}^\flat$, i.e. from an equivalence $\Conn\colon d_\mathcal{B}\log\mathcal{S}^\flat \to \Omega^{1,cl}_\mathcal{B}[1]$.

It transpires from the above description that locally defined $\mathcal{O}_{X/\mathcal{P}}^\times$-gerbes with flat $\mathcal{B}$-connective structure are objects of $(\mathcal{O}_{X/\mathcal{P}}^\times \xrightarrow{d_\mathcal{B}\log} \Omega^{1,cl}_\mathcal{B})[2]$. As a consequence, the group of equivalence classes of $\mathcal{O}_{X/\mathcal{P}}^\times$-gerbes with flat $\mathcal{B}$-connective structure is canonically isomorphic to $H^2(X;\mathcal{O}_{X/\mathcal{P}}^\times \xrightarrow{d_\mathcal{B}\log} \Omega^{1,cl}_\mathcal{B})$.

\subsection{Curving}\label{subsection: curving}
A \emph{curving} $\kappa$ on the a $\mathcal{B}$-connective structure $\Conn$ on a $\mathcal{O}_{X/\mathcal{P}}^\times$-gerbe $\mathcal{S}$ is a lift of the functor $\mathcal{S} \xrightarrow{\Conn} \Omega^1_\mathcal{B}[1]$ to $(\Omega^1_\mathcal{B} \to \Omega^2_\mathcal{B})[1]$, i.e. a factorization of $\Conn$ as
\begin{equation*}
\mathcal{S} \xrightarrow{\kappa} (\Omega^1_\mathcal{B} \to \Omega^2_\mathcal{B})[1] \to \Omega^1_\mathcal{B}[1] .
\end{equation*}
We refer to the functor $\mathcal{S} \xrightarrow{\kappa} (\Omega^1_\mathcal{B} \to \Omega^2_\mathcal{B})[1]$ as \emph{a connective structure with curving} (given by $\kappa$). 

Locally defined $\mathcal{O}_{X/\mathcal{P}}^\times$-gerbes with $\mathcal{B}$-connective structure with curving, i.e. pairs $(\mathcal{S}, \kappa)$ as above are objects of the Picard 2-stack which we denote $(\mathcal{O}_{X/\mathcal{P}}^\times \xrightarrow{d_\mathcal{B}\log} \Omega^1_\mathcal{B}\xrightarrow{d_\mathcal{B}}\Omega^2_\mathcal{B})[2]$. The group of equivalence classes of $\mathcal{O}_{X/\mathcal{P}}^\times$-gerbes with connective structure with curving is canonically isomorphic to $H^2(X;\mathcal{O}_{X/\mathcal{P}}^\times \xrightarrow{d_\mathcal{B}\log} \Omega^1_\mathcal{B} \xrightarrow{d_\mathcal{B}}\Omega^2_\mathcal{B})$.

The curving $\mathcal{S} \to (\Omega^1_\mathcal{B} \to \Omega^2_\mathcal{B})[1]$ is \emph{flat} if it factors through $(\Omega^1_\mathcal{B} \to \Omega^{2,cl}_\mathcal{B})[1]$. In this case we refer to the functor $\mathcal{S} \to (\Omega^1_\mathcal{B} \to \Omega^{2,cl}_\mathcal{B})[1]$ as a $\mathcal{B}$-connective structure with a flat curving. Locally defined $\mathcal{O}_{X/\mathcal{P}}^\times$-gerbes with $\mathcal{B}$-connective structure with flat curving are objects of the Picard 2-stack which we denote $(\mathcal{O}_{X/\mathcal{P}}^\times \xrightarrow{d_\mathcal{B}\log} \Omega^1_\mathcal{B}\xrightarrow{d_\mathcal{B}}\Omega^{2,cl}_\mathcal{B})[2]$.
The group of equivalence classes of $\mathcal{O}_{X/\mathcal{P}}^\times$-gerbes with connective structure with flat curving is canonically isomorphic to $H^2(X;\mathcal{O}_{X/\mathcal{P}}^\times \xrightarrow{d_\mathcal{B}\log} \Omega^1_\mathcal{B} \xrightarrow{d_\mathcal{B}}\Omega^{2,cl}_\mathcal{B})$.

The morphism of complexes $\Omega^{1,cl}_\mathcal{B} \to (\Omega^1_\mathcal{B} \to \Omega^{2,cl}_\mathcal{B})$ induces a canonical flat curving on a flat $\mathcal{B}$-connective structure.

\section{DQ-algebras}\label{section: DQ-algebras}
Throughout this section $X$ a $C^\infty$ manifold equipped with an integrable complex distribution $\mathcal{P}$, see Section \ref{section: distributions} for definitions and notations.

In the context of complex manifolds the notion of a DQ-algebra was introduced in \cite{KS}. 

\subsection{Star-products}
A \emph{star-product} on $\mathcal{O}_{X/\mathcal{P}}$ is a map
\begin{equation*}
\mathcal{O}_{X/\mathcal{P}}\otimes_\mathbb{C}\mathcal{O}_{X/\mathcal{P}} \to \mathcal{O}_{X/\mathcal{P}}[[t]]
\end{equation*}
of the form
\begin{equation}\label{star-product}
f\otimes g \mapsto f\star g = fg + \sum_{i=1}^\infty P_i(f,g)t^i,
\end{equation}
where $P_i$ are bi-differential operators. Such a map admits a unique $\mathbb{C}[[t]]$-bilinear extension
\begin{equation*}
\mathcal{O}_{X/\mathcal{P}}[[t]]\otimes_{\mathbb{C}[[t]]}\mathcal{O}_{X/\mathcal{P}}[[t]] \to \mathcal{O}_{X/\mathcal{P}}[[t]]
\end{equation*}
and the latter is required to define a structure of an associative unital $\mathbb{C}[[t]]$-algebra on $\mathcal{O}_{X/\mathcal{P}}[[t]]$.

\begin{prop}[(\cite{KS}, Proposition 2.2.3)]\label{prop: iso star-prod diff op}
Let $\star$ and $\star^\prime$ be star-products and let $\varphi\colon (\mathcal{O}_{X/\mathcal{P}}[[t]], \star) \to (\mathcal{O}_{X/\mathcal{P}}[[t]], \star^\prime)$ be a morphism of $\mathbb{C}[[t]]$-algebras. Then, there exists a unique sequence of differential operators $\{R_i \}_{i\geq 0}$ on $X$ such that $R_0 = 1$ and $\varphi(f) = \sum\limits_{i=0}^\infty R_i(f)t^i$ for any $f\in \mathcal{O}_{X/\mathcal{P}}$. In particular, $\varphi$ is an isomorphism.
\end{prop}

\begin{remark}
The paper \cite{KS} and, in particular, Proposition 2.2.3 of loc. cit. pertain to the holomorphic context, i.e. the case when $\mathcal{P}$ is a complex structure. However, it is easy to see that the proof of Proposition 2.2.3 as well as the results it is based upon carry over to the case of a general integrable distribution.
\end{remark}

For a star-product given by \eqref{star-product} the operation
\begin{equation}\label{star product assoc poisson}
f\otimes g \mapsto P_1(f,g) - P_1(g,f)
\end{equation}
is a Poisson bracket on $\mathcal{O}_{X/\mathcal{P}}$ which we refer to as \emph{the associated Poisson bracket}. It follows from Lemma \ref{lemma: isom dq same poisson} below that isomorphic star-products give rise to the same associated Poisson bracket.

\begin{definition}
A star-product given by \eqref{star-product} is said to be \emph{special} if $P_1$ is skew-symmetric.
\end{definition}

\begin{remark}
If the star-product is special, then $P_1(f,g) = \dfrac12\{f,g\}$, where the latter is the associated Poisson bracket.
\end{remark}

The following lemma is well-known.

\begin{lemma}\label{lemma: iso to special}
Any star-product is locally isomorphic to a special one.
\end{lemma}
\begin{proof}
Let $P_1^+$ (respectively, $P_1^-$) denote the skew-symmetrization (respectively, the symmetrization) of $P_1$. Associativity of the star-product implies that $P_1$, $P_1^+$ and $P_1^-$ are Hochschild cocycles of degree two on $\mathcal{O}_{X/\mathcal{P}}$. The Hochschild-Kostant-Rosenberg theorem says that, locally on $X$, $P_1^+$ is a bi-vector and  $P_1^-$ is a Hochschild coboundary, i.e., locally on $X$, there exists a differential operator $Q$ (acting on $\mathcal{O}_{X/\mathcal{P}}$) such that $P_1^-(f,g) = Q(fg) - Q(f)g - fQ(g)$. The star-product $\star^\prime$ defined by $f\star^\prime g = \exp(-tQ)(\exp(tQ)(f)\star\exp(tQ)(g)) = fg + P_1^{+}(f,g)t + \cdots$ is special.
\end{proof}

\subsection{DQ-algebras}
A DQ-algebra is a sheaf of $\mathbb{C}[[t]]$-algebras locally isomorphic to a star-product.
For a DQ-algebra $\mathbb{A}$ there is a canonical isomorphism $\mathbb{A}/t\cdot\mathbb{A}\cong \mathcal{O}_{X/\mathcal{P}}$. Therefore, there is a canonical map  (reduction modulo $t$) $\mathbb{A} \xrightarrow{\sigma} \mathcal{O}_{X/\mathcal{P}}$ of $\mathbb{C}[[t]]$-algebras.

A morphism of DQ-algebras is a morphism of sheaves of $\mathbb{C}[[t]]$-algebras.

For an open subset $U \subseteq X$ let $\DQA_{X/\mathcal{P}}(U)$ denote the category of DQ-algebras on $U$, where $U$ is equipped with the restriction of $\mathcal{P}$. The assignment $U \mapsto \DQA_{X/\mathcal{P}}(U)$ defines a stack on $X$ denoted $\DQA_{X/\mathcal{P}}$.

In what follows, we shall denote $\shHom_{\DQA_{X/\mathcal{P}}}$ simply by $\shHom_{\DQA}$.

\subsection{The associated Poisson structure}
Suppose that $\mathbb{A}$ is a DQ-algebra. The composition
\begin{equation*}
\mathbb{A}\otimes\mathbb{A} \xrightarrow{[.,.]} \mathbb{A} \xrightarrow{\sigma} \mathcal{O}_{X/\mathcal{P}}
\end{equation*}
is trivial. Therefore, the commutator $\mathbb{A}\otimes\mathbb{A} \xrightarrow{[.,.]} \mathbb{A}$ takes values in $t\mathbb{A}$. The composition
\begin{equation*}
\mathbb{A}\otimes\mathbb{A} \xrightarrow{[.,.]} t\mathbb{A} \xrightarrow{t^{-1}} \mathbb{A} \xrightarrow{\sigma} \mathcal{O}_{X/\mathcal{P}}
\end{equation*}
factors uniquely as
\begin{equation*}
\mathbb{A}\otimes\mathbb{A} \xrightarrow{\sigma\otimes\sigma} \mathcal{O}_{X/\mathcal{P}}\otimes\mathcal{O}_{X/\mathcal{P}} \xrightarrow{\{.,.\}} \mathcal{O}_{X/\mathcal{P}}
\end{equation*}
The latter map, $\{.,.\}\colon \mathcal{O}_{X/\mathcal{P}}\otimes\mathcal{O}_{X/\mathcal{P}} \to \mathcal{O}_{X/\mathcal{P}}$ is a Poisson bracket on $\mathcal{O}_{X/\mathcal{P}}$, hence corresponds to a bi-vector $\poisson\in \Gamma(X;\bigwedge^2\mathcal{T}_{X/\mathcal{P}})$. If $\mathbb{A}$ is a star-product we recover the \eqref{star product assoc poisson}.

\begin{lemma}\label{lemma: isom dq same poisson}
Locally isomorphic DQ-algebras give rise to the same associated Poisson bracket.
\end{lemma}
\begin{proof}
Suppose that $\mathbb{A}_i$, $i=1,2$ is a DQ-algebras with associated Poisson bracket $\{.,.\}_i$ and $\Phi \colon \mathbb{A}_1  \to \mathbb{A}_2$ is a morphism of such. Let $f,g \in \mathcal{O}_{X/\mathcal{P}}$, $\widetilde{f}, \widetilde{g} \in \mathbb{A}_1$ with $f = \widetilde{f} + t\mathbb{A}_1$, $g = \widetilde{g} + t\mathbb{A}_1$. 

Since $\Phi$ induces the identity map modulo $t$,
\[
t\cdot\{f,g\}_2 + t^2\mathbb{A}_2 = \Phi(\widetilde{f})\Phi(\widetilde{g}) -\Phi(\widetilde{g})\Phi(\widetilde{f}) + t^2\mathbb{A}_2 = \Phi(\widetilde{f}\widetilde{g} - \widetilde{g}\widetilde{f} + t^2\mathbb{A}_1) = t\cdot\{f,g\}_1 + t^2\mathbb{A}_2
\]
Thus, if $\mathbb{A}_i$, $i=1,2$ are locally isomorphic, then the associated Poisson brackets are locally equal hence equal.
\end{proof}

\subsection{Standard sections}
Let $\mathbb{A}$ be a DQ-algebra. Recall (\cite{KS}, Definition 2.2.6) that a $\mathbb{C}$-linear section $\phi\colon \mathcal{O}_{X/\mathcal{P}} \to \mathbb{A}$ of the map $\mathbb{A} \xrightarrow{\sigma} \mathcal{O}_{X/\mathcal{P}}$ is called \emph{standard} if there exist bi-differential operators $P_i$, $i=0,1,\ldots$, such that for any $f, g\in\mathcal{O}_{X/\mathcal{P}}$
\begin{equation}\label{P}
\phi(f)\phi(g) = \phi(fg) + \sum_{i\geq 1}\phi(P_i(f,g))t^i,
\end{equation}
where the left-hand side product is computed in $\mathbb{A}$.
In this case
\begin{equation*}
f\otimes g\mapsto f\star_\phi g := fg + \sum_{i\geq 1}P_i(f,g)t^i
\end{equation*}
defines a star-product.

A standard section $\phi\colon \mathcal{O}_{X/\mathcal{P}} \to \mathbb{A}$ extends by $t$-linearity to a morphism of DQ-algebras $\widetilde{\phi} \colon (\mathcal{O}_{X/\mathcal{P}}[[t]], \star_\phi) \to \mathbb{A}$. Conversely, a morphism of DQ-algebras $\varphi \colon (\mathcal{O}_{X/\mathcal{P}}[[t]], \star) \to \mathbb{A}$ restricts to a standard section $\phi := \varphi\vert_{\mathcal{O}_{X/\mathcal{P}}}$ such that $\star_\phi = \star$ and $\widetilde{\phi} = \varphi$.

We will call a standard section $\phi$ \emph{special} if the corresponding star-product $\star_\phi$ is special.

\begin{notation}
Let $\Sigma(\mathbb{A})$ denote the sheaf of locally defined special standard sections.\qed
\end{notation}

For $U\subseteq X$ and open subset, $\phi_1, \phi_2 \in \Sigma(\mathbb{A})(U)$ and $k = 1,2,\ldots$ we say that $\phi_1$ and $\phi_2$ are \emph{equivalent modulo $t^{k+1}$} if the compositions $\mathcal{O}_{X/\mathcal{P}}(U) \xrightarrow{\phi_i} \mathbb{A}(U) \to \mathbb{A}(U)/t^{k+1}\mathbb{A}(U)$, $i=1,2$, coincide. Clearly, ``equivalence modulo $t^{k+1}$" is an equivalence relation on $\Sigma(\mathbb{A})$ which we denote $\sim_k$. We denote by $\Sigma_k(\mathbb{A})$ the sheaf associated to the presheaf $U \mapsto \Sigma(\mathbb{A})(U)/\sim_k$.

\begin{prop}\label{prop: special std sec exist locally}
Let $\mathbb{A}$ be a DQ-algebra.
\begin{enumerate}
\item The sheaf $\Sigma(\mathbb{A})$ is locally non-empty.

\item The quotient map $\Sigma(\mathbb{A}) \to \Sigma_k(\mathbb{A})$ is locally surjective on sets of sections.

\item For a morphism $\Phi \colon \mathbb{A}_0 \to \mathbb{A}_1$ of DQ-algebras and $\phi \in \Sigma(\mathbb{A}_0)$ the composition $\Phi\circ\phi$ is a special standard section, i.e. $\Phi\circ\phi \in \Sigma(\mathbb{A}_1)$.
\end{enumerate}
\end{prop}
\begin{proof}
Since the question is local we may assume that $\mathbb{A} = (\mathcal{O}_{X/\mathcal{P}}[[t]], \star)$ is a star product given by \eqref{star-product}.
\begin{enumerate}
\item It follows from Lemma \ref{lemma: iso to special} that locally there exists an isomorphism 
$\varphi \colon (\mathcal{O}_{X/\mathcal{P}}[[t]], \star) \to \mathbb{A}$ with $\star$ a special star-product. Then $\varphi\vert_{\mathcal{O}_{X/\mathcal{P}}}$ is a special standard section.

\item Since the question is local we may assume that the map $q \colon \mathbb{A} \to \mathbb{A}/t^{k+1}\mathbb{A}$ admits a splitting $s \colon \mathbb{A}/t^{k+1}\mathbb{A} \to \mathbb{A}$, $q\circ s = \id$. The composition $s \circ q \colon \mathbb{A}  \to \mathbb{A}$ preserves equivalence modulo $t^{k+1}$, hence induces the map $s \circ q \colon \Sigma_k(\mathbb{A}) \to \Sigma(\mathbb{A})$. Since the composition $\Sigma_k(\mathbb{A}) \to \Sigma(\mathbb{A}) \to \Sigma_k(\mathbb{A})$ is the identity map the claim follows.

\item Since $\Phi$ induces the identity map modulo $t$ the composition $\Phi\circ\phi$ is a section. By Proposition \ref{prop: iso star-prod diff op} the terms of
\[
(\Phi\circ\phi)(f)(\Phi\circ\phi)(g) = (\Phi\circ\phi)(fg) + \sum_{i\geq 1}(\Phi\circ\phi)(P_i(f,g))t^i,
\]
are bi-differential operators and therefore the section $\Phi\circ\phi$ is standard if $\phi$ is. The formula shows that, if $\phi$ is special, then so is $\Phi\circ\phi$.
\end{enumerate}
\end{proof}

Thus, the assignment $\mathbb{A} \mapsto \Sigma(\mathbb{A})$, $(\mathbb{A}_0 \xrightarrow{\Phi}  \mathbb{A}_1) \mapsto (\Sigma(\Phi) \colon \phi \mapsto \Phi\circ\phi)$ defines a functor, denoted $\Sigma$, on the category of DQ-algebras.

It is clear that the map $\Sigma(\Phi) \colon \Sigma(\mathbb{A}_0) \to \Sigma(\mathbb{A}_1)$ induced by the morphism of DQ-algebras $\Phi \colon \mathbb{A}_0 \to \mathbb{A}_1$ preserves equivalence modulo $t^k$. Thus, the functor $\Sigma$ gives rise to functors $\Sigma_k$ for all $k = 1,2,\ldots$.

\subsection{The functor $\Sigma_1$}
Suppose that $\mathbb{A}$ is a DQ-algebra with the associated Poisson tensor $\pi$ and the corresponding Lie algebroid (structure on $\Omega^1_{X/\mathcal{P}}$) denoted $\Pi$.

\begin{lemma}\label{lemma: change of std sec special}
Suppose that $\phi_0$ and  $\phi_1$ are standard sections of $\mathbb{A}$ with the corresponding star-products $\star_{(j)} := \star_{\phi_i}$ given by
\[
f\star_{(j)} g = fg + P^{(j)}_1(f,g)t + P^{(j)}_2(f,g)t^2 + \cdots
\]
for $j=0,1$. Let $R = 1 + \sum\limits_{i=1}^\infty R_i t^i$, $R_i$ differential operators, denote the solution of $\widetilde{\phi_1} = \widetilde{\phi_0}\circ R$ uniquely determined by Proposition \ref{prop: iso star-prod diff op}. Then, $P^{(0)}_1 = P^{(1)}_1$ if and only if $R_1$ is a derivation, i.e. a section of $\mathcal{T}_{X/\mathcal{P}}$.
\end{lemma}
\begin{proof}
The operator $R$ defines a morphism of algebras $(\mathcal{O}_{X/\mathcal{P}}[[t]], \star_{(1)}) \to (\mathcal{O}_{X/\mathcal{P}}[[t]], \star_{(0)})$, i.e.
\[
R(f) \star_{(0)} R(g) = R(f \star_{(1)} g) .
\]
Comparing the calculations
\[
R(f) \star_{(0)} R(g) = fg + (fR_1(g) + R_1(f)g + P^{(0)}_1(f,g))t +  \dots .
\]
and
\[
R(f \star_{(1)} g) = fg + (R_1(fg) + P^{(1)}_1(f,g))t + \dots .
\]
one concludes that
\[
P^{(0)}_1(f,g) = P^{(1)}_1(f,g) \text{ if and only if } fR_1(g) + R_1(f)g = R_1(fg) ,
\]
i.e. $R_1$ is a derivation.
\end{proof}

\begin{cor}\label{cor: change of std sec special}
In the notation of Lemma \ref{lemma: change of std sec special}, suppose in addition that $\phi_0$ is special. Then, $\phi_1$ is special if and only if $R_1$ is a derivation, i.e. a section of $\mathcal{T}_{X/\mathcal{P}}$.
\end{cor}

For $\xi\in\mathcal{T}_{X/\mathcal{P}} = \Omega^1_\Pi$ let $R_\xi := 1 + \xi t$. Corollary \ref{cor: change of std sec special} implies that for  $\phi \in  \Sigma(\mathbb{A})$ the section $\widetilde{\phi}\circ R_\xi\vert_{\mathcal{O}_{X/\mathcal{P}}}$ is special, i.e. $\phi \mapsto \widetilde{\phi}\circ R_\xi\vert_{\mathcal{O}_{X/\mathcal{P}}}$ gives rise to a map $\Sigma(\mathbb{A}) \to \Sigma(\mathbb{A})$.

Suppose that $\phi_0, \phi_1 \in \Sigma(\mathbb{A})$ are equivalent modulo $t^2$, which is to say $\phi_0(f) - \phi_1(f) \in t^2\mathbb{A}$ for any $f \in \mathcal{O}_{X/\mathcal{P}}$ which implies that $\widetilde{\phi_0}(f) - \widetilde{\phi_1}(f) \in t^2\mathbb{A}$ for any $f \in \mathcal{O}_{X/\mathcal{P}}[[t]]$. Then, for any $f \in \mathcal{O}_{X/\mathcal{P}}$, $\widetilde{\phi_0}\circ R_\xi\vert_{\mathcal{O}_{X/\mathcal{P}}}(f) - \widetilde{\phi_1}\circ R_\xi\vert_{\mathcal{O}_{X/\mathcal{P}}}(f) =  \widetilde{\phi_0}( R_\xi\vert_{\mathcal{O}_{X/\mathcal{P}}}(f)) - \widetilde{\phi_1}(R_\xi\vert_{\mathcal{O}_{X/\mathcal{P}}}(f)) \in t^2\mathbb{A}$.
Therefore, the map $\phi \mapsto \widetilde{\phi}\circ R_\xi\vert_{\mathcal{O}_{X/\mathcal{P}}}$ preserves equivalence modulo $t^2$, hence descends to a map  $\Sigma_1(\mathbb{A}) \to \Sigma_1(\mathbb{A})$ denoted $\phi \mapsto \phi + \xi$.

\begin{lemma}\label{lemma: sp std sec T-torsor}
In the notation introduced above
\begin{enumerate}
\item The map $\Sigma_1(\mathbb{A})\times\Omega^1_\Pi \to \Sigma_1(\mathbb{A})$ given by $(\phi,\xi) \mapsto \phi + \xi$ defines a free action of the group $\Omega^1_\Pi$ on the sheaf $\Sigma_1(\mathbb{A})$. Moreover, $\Sigma_1(\mathbb{A})$ is a torsor under (the above action of) $\Omega^1_\Pi$.

\item For a morphism of DQ-algebras $\Phi \colon \mathbb{A}_0 \to \mathbb{A}_1$ the induced map $\Sigma_1(\Phi) \colon \Sigma_1(\mathbb{A}_0) \to \Sigma_1(\mathbb{A}_1)$ is a morphism of $\Omega^1_\Pi$-torsors.

\item \label{lemma: sp std sec T-torsor.3} For $a \in \mathbb{A}^\times$ let $\Ad a$ denote the inner automorphism defined by $(\Ad a) (f)= a f a^{-1}$. Then, for $\phi \in \Sigma_1(\mathbb{A})$
\[
\Sigma_1(\Ad a)(\phi) = \phi+ d_\Pi\log(\sigma(a)).
\]
\end{enumerate}
\end{lemma}
\begin{proof}
{~}
\begin{enumerate}
\item
The sheaf $\Sigma_1(\mathbb{A})$ is locally non-empty by Proposition \ref{prop: special std sec exist locally}.

Note that, for $\xi, \eta \in \mathcal{T}_{X/\mathcal{P}}$ the equality $R_\xi \circ R_\eta = R_{\xi + \eta} + t^2\xi\circ\eta$ holds. Therefore, the formula $(\phi,\xi) \mapsto \phi + \xi$ does indeed define an action of the group $\Omega^1_\Pi$. Since $\widetilde{\phi}(R_\xi(f)) - \widetilde{\phi}(f) = t\phi(\xi(f)) + t^2\mathbb{A}$, it follows that the action is free. Corollary \ref{cor: change of std sec special} implies that the action is in fact transitive.

\item
By associativity of composition
\[
\Phi(\widetilde{\phi}\circ R_\xi(f)) = \widetilde{(\Phi\circ\phi)}\circ R_\xi(f)
\]
for all $\phi \in \Sigma(\mathbb{A}_0)$, $\xi \in \mathcal{T}_{X/\mathcal{P}}$ and $f \in \mathcal{O}_{X/\mathcal{P}}$. Therefore,
\[
\Sigma_1(\phi + \xi) = \Sigma_1(\phi) + \xi .
\]

\item Since the statement is local, we can assume that $a = \exp \alpha$, $\alpha \in \mathbb{A}$. Note that, for $b \in \mathbb{A}$,
\[
\Ad(\exp \alpha) (b) = \sum_{i=0}^{\infty} \frac{(\ad \alpha)^i(b)}{i!}=  b +t\{\sigma(\alpha), \sigma(b)\}+t^2 \mathbb{A},
\]
where $(\ad \alpha)(b):= [\alpha, b]$. Therefore, for $f \in \mathcal{O}_{X/\mathcal{P}}$, since $\sigma\circ\phi = \id$,
\begin{multline*}
((\Ad a) \circ \phi)(f) = \phi(f) + t\{\sigma(\alpha), \sigma(\phi(f))\}+t^2 \mathbb{A} = \\
\phi(f) + t d_\Pi\log(\sigma(a))(f) + t^2 \mathbb{A} = \widetilde{\phi}\circ R_{d_\Pi\log(\sigma(a))}(f) + t^2 \mathbb{A}
\end{multline*}
\end{enumerate}
\end{proof}

In view of Lemma \ref{lemma: sp std sec T-torsor}, the assignment $\mathbb{A} \mapsto \Sigma_1(\mathbb{A})$, $\Phi \mapsto \Sigma_1(\Phi)$, defines a morphism of stacks
\[
\Sigma_1 \colon \DQA_{X/\mathcal{P}} \to \Omega^1_\Pi[1] .
\]

\subsection{Subprincipal symbols}\label{subsection: Subprincipal symbols}
The construction presented below can be traced back to \cite{V75}. Subprincipal curvature defined below appears in the context of classification of star-products in \cite{BCG97} and in \cite{B02}, where it is called "semiclassical curvature".

Multiplication by $t^n$ induces the isomorphism $\mathcal{O}_{X/\mathcal{P}} \xrightarrow{\cong} t^n\mathbb{A}/t^{n+1}\mathbb{A}$. In particular, there is a short exact sequence
\begin{equation}\label{subprincipal ext}
0 \to \mathcal{O}_{X/\mathcal{P}} \xrightarrow{\cdot t} \mathbb{A}/t^2\mathbb{A} \to \mathcal{O}_{X/\mathcal{P}} \to 0
\end{equation}
Let $\widetilde{\{\ ,\ \}}$ denote operation on $\mathbb{A}/t^2\mathbb{A}$ induced by $\dfrac1t[\ ,\ ]$. The operation $\widetilde{\{\ ,\ \}}$ endows $\mathbb{A}/t^2\mathbb{A}$ with a structure of a Lie algebra so that the exact sequence \eqref{subprincipal ext} exhibits $\mathbb{A}/t^2\mathbb{A}$ as an abelian extension of (the Lie algebra) $\mathcal{O}_{X/\mathcal{P}}$ equipped with the associated Poisson bracket.

\begin{lemma}\label{lemma: curvature divisible by t}
For $\phi \in \Sigma(\mathbb{A})$ and $f,g \in \mathcal{O}_{X/\mathcal{P}}$
\[
  \widetilde{\{\phi(f) ,\phi(g) \}} - \phi(\{f,g\}) \in t\mathbb{A}.
\]
\end{lemma}
\begin{proof}
 Since (in the notations of \eqref{P})
\[
[\phi(f), \phi(g)] -( t \phi(\{f, g\})+ t^2\phi(P_2(f,g) -P_2(g, f))) \in t^3 \mathbb{A} ,
\]
it follows that
\[
\widetilde{\{\phi(f) ,\phi(g) \}} - \phi(\{f,g\})= t \phi(P_2(f,g) -P_2(g, f)) +t^2\mathbb{A} \in t\mathbb{A}
\]
\end{proof}

By Lemma \ref{lemma: curvature divisible by t}
\[
\widetilde{\{\phi(f) ,\phi(g) \}} - \phi(\{f,g\}) + t^2\mathbb{A} \in t\mathbb{A}/\mathbb{A}/t^2\mathbb{A} \xleftarrow{\cdot t}\mathcal{O}_{X/\mathcal{P}} ,
\]
For $\phi \in \Sigma(\mathbb{A})$ we define the map
\begin{equation}\label{subprincipal curvature}
c(\phi) \colon \mathcal{O}_{X/\mathcal{P}}\otimes_\mathbb{C}\mathcal{O}_{X/\mathcal{P}} \to \mathcal{O}_{X/\mathcal{P}}
\end{equation}
by
\[
tc(\phi)(f,g) =  \widetilde{\{\phi(f) ,\phi(g) \}} - \phi(\{f,g\}) + t^2 \mathbb{A} \in t\mathbb{A}/t^2\mathbb{A} .
\]

\begin{lemma}\label{lemma: subprincipal curvature natural}
Let  $\Phi \colon \mathbb{A}_0 \to \mathbb{A}_1$ be   a morphism of DQ-algebras. Then
\[
c(\Phi \circ \phi) = c(\phi)
\]
\end{lemma}
\begin{proof}
We have
\begin{multline*}
tc(\Phi \circ \phi)(f,g) = \widetilde{\{\Phi \circ \phi(f) ,\Phi \circ \phi(g) \}} - \Phi \circ \phi(\{f,g\}) + t^2 \mathbb{A}_1 = \\
\Phi \left(\widetilde{\{\phi(f) ,\phi(g) \}} - \phi(\{f,g\}) + t^2 \mathbb{A}_0\right) = \Phi(tc(\phi)(f,g).
\end{multline*}
The statement follows from the commutativity of the diagram
\[
\begin{CD}
\mathcal{O}_{X/\mathcal{P}} @>{\cdot t}>> t\mathbb{A}_0/t^2\mathbb{A}_0 \\
@| @VV{\Phi}V \\
\mathcal{O}_{X/\mathcal{P}} @>{\cdot t}>> t\mathbb{A}_1/t^2\mathbb{A}_1
\end{CD}
\]
\end{proof}

\begin{prop}\label{prop: subprincipal curvature}
{~}
\begin{enumerate}
\item \label{subprincipal curvature 2}The map \eqref{subprincipal curvature} is a skew-symmetric bi-derivation, hence $c(\phi) \in  \bigwedge^2\mathcal{T}_{X/\mathcal{P}}$.

\item \label{subprincipal curvature 3}The bi-vector $c(\phi)$ satisfies $d_\Pi c(\phi) = 0$.

\item \label{subprincipal curvature 5}The map $c(\phi)$ depends only on the equivalence class of $\phi$ modulo $t^2$.

\item \label{subprincipal curvature 4}For $\xi\in\mathcal{T}_{X/\mathcal{P}}$ 
\[
c(\phi + \xi) = c(\phi) + d_\Pi \xi .
\]
\end{enumerate}
\end{prop}
\begin{proof}
\begin{enumerate}
\item The skew-symmetry of $c(\phi)$ is clear. Applying $(\widetilde{\phi})^{-1}$ to the identity
\[
[\phi(f), \phi(g) \phi(h)] = \phi(g)[\phi(f),  \phi(h)]+ [\phi(f), \phi(g)] \phi(h) ,
\]
expanding the result in powers of $t$ and comparing the coefficients of $t^2$ we obtain
\[
c(\phi)(f, gh) + \frac{1}{2}\{f, \{g, h\}\} = g c(\phi)(f, h)+\frac{1}{2}\{g, \{f, h\}\} + c(\phi)(f, g)h +\frac{1}{2} \{\{f,g\}, h\},
\]
which implies that
\[
c(\phi)(f, gh)   = g c(\phi)(f, h)  + c(\phi)(f, g)h
\]
i.e. $c(\phi)$ is a derivation in the second variable. Skew-symmetry implies that it is a biderivation.

\item 
Applying $(\widetilde{\phi})^{-1}$ to the identity
\[
[\phi(f), [\phi(g), \phi(h)]] = [\phi(g),[\phi(f),  \phi(h)]]+ [[\phi(f), \phi(g)], \phi(h)] ,
\]
expanding the result in powers of $t$ and comparing the coefficients of $t^3$ we obtain
\[
c(\phi)(f, \{g, h\})  + \{f, c(\phi)(g, h)\} = c(\phi)(g, \{f, h\})  + \{g, c(\phi)(f, h)\}+ c(\phi)(\{f, g\}, h)  + \{ c(\phi)(f, g), h\}
\]
which is equivalent to $d_\Pi(\phi) = 0$.

\item If $\phi, \phi' \in \Sigma(\mathbb{A})$, then there exists $R = 1 +  \sum_{i=1}^\infty R_i t^i$, where $R_i$ are differential operators and $R_1 = \xi$ is a vector field such that $\phi' = \phi \circ R$. Direct calculation shows that
\begin{multline}\label{change of section}
[\phi'(f), \phi'(g)] = \phi (\{f,g\}t +
( \{\xi(f),g\} + \{f,\xi(g)\}  + c(\phi)(f,g))t^2 + \dots ).
\end{multline}
Hence,
\begin{multline*}
\frac{1}{t}[\phi'(f), \phi'(g)] - \phi'(\{f,g\}) \\=  \phi (\{f,g\}) +
\phi( \{\xi(f),g\} + \{f,\xi(g)\}  + c(\phi)(f,g))t - \phi'(\{f,g\})
\\ = \phi( \{\xi(f),g\} + \{f,\xi(g)\} - \xi(\{f, g\})  + c(\phi)(f,g))t + t^2\mathbb{A}
\end{multline*}
and
\begin{equation}\label{change of section subprin curv}
c(\phi')(f,g)= c(\phi)(f,g) + \{\xi(f),g\} + \{f,\xi(g)\} - \xi(\{f, g\}) .
\end{equation}
The sections $\phi$ and $\phi'$ are equivalent modulo $t^2$ if and only if $\xi = 0$ in which case \eqref{change of section subprin curv} implies $c(\phi') = c(\phi)$.

\item Note that \eqref{change of section subprin curv} is equivalent to $c(\phi')(f,g) = c(\phi)(f,g) + d_\Pi\xi(f,g)$.
\end{enumerate}
\end{proof}

According to parts (\ref{subprincipal curvature 2}), (\ref{subprincipal curvature 3}) and (\ref{subprincipal curvature 5}) of Proposition \ref{prop: subprincipal curvature} the assignment $c \colon \phi \mapsto c(\phi)$ descends to a map
\begin{equation}\label{subprincipal curvature map}
c \colon \Sigma_1(\mathbb{A}) \to \Omega^{2,cl}_\Pi .
\end{equation}
which, according to (\ref{subprincipal curvature 4}) of Proposition \ref{prop: subprincipal curvature} is a morphism of $\Omega^1_\Pi$-torsors compatible with the map of sheaves of groups $d_\Pi \colon \Omega^1_\Pi \to \Omega^{2,cl}_\Pi$. In other words, the map \eqref{subprincipal curvature map} endows $ \Sigma_1(\mathbb{A})$ with a canonical structure of a $(\Omega^1_\Pi \xrightarrow{d_\Pi} \Omega^{2,cl}_\Pi)$-torsor.

\begin{definition}
We refer to \eqref{subprincipal curvature map} as the \emph{subprincipal curvature} map.
\end{definition}

Lemma \ref{lemma: subprincipal curvature natural} and Lemma \ref{lemma: subprincipal curvature natural} imply that the assignment $\mathbb{A} \mapsto \Sigma_1(\mathbb{A})$, $\Phi \mapsto \Sigma_1(\Phi)$ defines a morphism of stacks
\[
\Sigma_1 \colon \DQA_{X/\mathcal{P}} \to (\Omega^1_\Pi \xrightarrow{d_\Pi} \Omega^{2,cl}_\Pi)[1] .
\]

\subsection{Bi-invertible bi-modules}\label{subsection: Bi-modules}
Suppose that $\mathbb{A}_0$ and $\mathbb{A}_1$ are DQ-algebras.

\begin{definition}[\cite{KS}, Definition 2.1.6]
$\mathbb{A}_1\otimes_{\mathbb{C}[[t]]}\mathbb{A}_0^\op$-module $\mathbb{A}_{01}$ is called \emph{bi-invertible} if locally on $X$ there exists a section $b\in\mathbb{A}_{01}$ such that the map $\mathbb{A}_1 \ni a \mapsto (a\otimes 1)b \in \mathbb{A}_{01}$ (respectively, $\mathbb{A}_0 \ni a \mapsto (1\otimes a)b \in \mathbb{A}_{01}$) is an isomorphism of $\mathbb{A}_1$- (respectively, $\mathbb{A}_0^\op$-) modules.
\end{definition}

\begin{lemma}\label{lemma: bi-inv same poisson}
Suppose that there exists a bi-invertible $\mathbb{A}_1\otimes_{\mathbb{C}[[t]]}\mathbb{A}_0^\op$-module.
Then, the associated Poisson bi-vectors of $\mathbb{A}_0$ and $\mathbb{A}_1$ are equal.
\end{lemma}
\begin{proof}
Existence of a bi-invertible module implies that the algebras are locally isomorphic. It follows from Lemma \ref{lemma: isom dq same poisson} that locally isomorphic algebras give rise to the same Poisson bi-vector.
\end{proof}

In what follows we suppose that $\mathbb{A}_{01}$ is a bi-invertible $\mathbb{A}_1\otimes_{\mathbb{C}[[t]]}\mathbb{A}_0^\op$-module. Let $\gr\mathbb{A}_{01} = \mathbb{A}_{01}/t\mathbb{A}_{01}$ and let $\gr \colon \mathbb{A}_{01} \to \gr\mathbb{A}_{01}$ denote the canonical projection. Let $\poisson$ denote the common Poisson bi-vector of $\mathbb{A}_0$ and $\mathbb{A}_1$ and let $\Pi$ denote the corresponding Lie algebroid (structure on $\Omega^1_{X/\mathcal{P}}$, see \ref{subsubsection: Poisson structures}).

Let $\mathbb{A}_{01}^\times$ denote the subsheaf of $\mathbb{A}_{01}$ consisting of sections $b \in \mathbb{A}_{01}$ such that the maps $\mathbb{A}_1 \ni a \mapsto (a\otimes 1)b \in \mathbb{A}_{01}$ and $\mathbb{A}_0 \ni a \mapsto (1\otimes a)b \in \mathbb{A}_{01}$) are isomorphisms of $\mathbb{A}_1$- (respectively, $\mathbb{A}_0^\op$-) modules. For $b \in \mathbb{A}_{01}^\times$ let
\[
\Phi_b \colon \mathbb{A}_0 \to \mathbb{A}_1
\]
denote the morphism of DQ-algebras uniquely determined by
\[
(\Phi_b(a)\otimes 1)\cdot b = (1\otimes a)b .
\]
The map $\gr \colon \mathbb{A}_{01} \to \gr\mathbb{A}_{01}$ restricts to $\gr \colon \mathbb{A}_{01}^\times \to (\gr\mathbb{A}_{01})^\times$.

The assignment $b \mapsto \Phi_b$ defines a map
\begin{equation}\label{inv bimod to hom alg}
\mathbb{A}_{01}^\times \to \shHom_{\DQA}(\mathbb{A}_0, \mathbb{A}_1),
\end{equation}
while the functor $\Sigma_1$ give rise to the map
\[
\Sigma_1 \colon \shHom_{\DQA}(\mathbb{A}_0, \mathbb{A}_1) \to \shHom_{(\Omega_\Pi^1 \xrightarrow{d_\Pi} \Omega_\Pi^{2,cl})[1]}(\Sigma_1(\mathbb{A}_0), \Sigma_1(\mathbb{A}_1)).
\]

\begin{lemma}\label{lemma: bimod preserve curv}
{~}
\begin{enumerate}
\item The composition
\[
\mathbb{A}_{01}^\times \xrightarrow{\eqref{inv bimod to hom alg}} \shHom_{\DQA}(\mathbb{A}_0, \mathbb{A}_1) \xrightarrow{\Sigma_1} \shHom_{(\Omega_\Pi^1 \xrightarrow{d_\Pi} \Omega_\Pi^{2,cl})[1]}(\Sigma_1(\mathbb{A}_0), \Sigma_1(\mathbb{A}_1))
\]
factors through $(\gr\mathbb{A}_{01})^\times$.

\item The induced map
\begin{equation}\label{sigma one bimod}
(\gr\mathbb{A}_{01})^\times \to \shHom_{(\Omega_\Pi^1 \xrightarrow{d_\Pi} \Omega_\Pi^{2,cl})[1]}(\Sigma_1(\mathbb{A}_0), \Sigma_1(\mathbb{A}_1))
\end{equation}
is a morphism of torsors compatible with the morphism of groups $d_\Pi\log \colon \mathcal{O}^\times_{X/\mathcal{P}} \to \Omega_\Pi^1$.
\end{enumerate}
\end{lemma}
\begin{proof}
A section $b\in\mathbb{A}_{01}^\times$ determines the isomorphism $\Phi_b \colon \mathbb{A}_0 \to \mathbb{A}_1$  and a compatible
isomorphism of $\mathbb{A}_{01}$ with $\mathbb{A}_1$ viewed as $\mathbb{A}_1\otimes_{\mathbb{C}[[t]]}\mathbb{A}_1^\op$-module.
Functoriality of $\Sigma_1$ shows that it is sufficient to prove the statement in the case when $\mathbb{A}_0= \mathbb{A}_1= \mathbb{A}$, and $\mathbb{A}_{01}= \mathbb{A}$ as an $\mathbb{A}$-bi-module. In this case the morphism
\[
\mathbb{A} ^\times \to \shHom_{\DQA}(\mathbb{A}, \mathbb{A})
\]
is given by $a \mapsto \Ad a$, and both parts of the lemma follow from the formula for action of inner automorphisms 
on $\Sigma_1(\mathbb{A})$ (Lemma \ref{lemma: sp std sec T-torsor}, part (\ref{lemma: sp std sec T-torsor.3})):

\[
\Sigma_1(\Ad a)(\phi) = \phi+ d_\Pi\log(\sigma(a)),\ \phi \in \Sigma_1(\mathbb{A}).
\]

.
\end{proof}

\begin{notation}
The map \eqref{sigma one bimod} will be denoted $\Sigma_1(\mathbb{A}_{01})$.
\end{notation}

The proof of the following lemma is left to the reader.

\begin{lemma}\label{lemma: Sigma 1 map of bimodules}
Suppose that $\Psi \colon \mathbb{A}_{01}^{(1)} \to \mathbb{A}_{01}^{(0)}$ is an isomorphism of bi-invertible $\mathbb{A}_1\otimes_{\mathbb{C}[[t]]}\mathbb{A}_0^\op$-modules. Then $\Sigma_1(\mathbb{A}_{01}^{(1)}) = (\gr\Psi)^\times \circ \Sigma_1(\mathbb{A}_{01}^{(0)})$
\end{lemma}

\subsection{Quasi-classical limits of bi-modules}\label{subsection: Quasi-classical limits of bi-modules}
The canonical contravariant connection on the classical limit of a bi-module deformation was constructed in \cite{B02} (see also \cite{BW04}) in the setting of star-products. The construction presented below is a generalization of that of loc. cit. to the case of DQ-algebras.

Suppose that $\mathbb{A}_{01}$ is a bi-invertible $\mathbb{A}_1\otimes_{\mathbb{C}[[t]]}\mathbb{A}_0^\op$-module. By Lemma \ref{lemma: bi-inv same poisson} $\mathbb{A}_0$ and $\mathbb{A}_1$ give rise to the same Poisson bi-vector $\poisson$. Let $\Pi$ denote the corresponding Lie algebroid (see \ref{subsubsection: Poisson structures}).

Then, $\gr\mathbb{A}_{01}$ has a canonical structure of a $\mathcal{O}_{X/\mathcal{P}}$-bi-module which is central, i.e. the two $\mathcal{O}_{X/\mathcal{P}}$-module structures coincide. Moreover, $\gr\mathbb{A}_{01}$ is locally free of rank one over $\mathcal{O}_{X/\mathcal{P}}$.

Therefore, the composition
\begin{equation*}
\mathbb{A}_1\times_{\mathcal{O}_{X/\mathcal{P}}}\mathbb{A}_0\otimes \mathbb{A}_{01} \to \mathbb{A}_{01} \xrightarrow{\gr} \gr\mathbb{A}_{01}
\end{equation*}
defined by $(a_1,a_0)\otimes b \mapsto (a_1b - ba_0) + t\mathbb{A}_{01}$ is trivial. Thus, the first map factors through the inclusion $t\mathbb{A}_{01} \hookrightarrow \mathbb{A}_{01}$.
The composition
\begin{equation*}
\mathbb{A}_1\times_{\mathcal{O}_{X/\mathcal{P}}}\mathbb{A}_0\otimes \mathbb{A}_{01} \to t\mathbb{A}_{01} \xrightarrow{t^{-1}} \mathbb{A}_{01} \to \gr\mathbb{A}_{01}
\end{equation*}
factors through a unique map
\begin{equation*}
\mathbb{A}_1/t^2\mathbb{A}_1\times_{\mathcal{O}_{X/\mathcal{P}}}\mathbb{A}_0/t^2 \mathbb{A}_0 \otimes \gr\mathbb{A}_{01} \to \gr\mathbb{A}_{01}
\end{equation*}
which gives rise to the map
\begin{equation}\label{action of fib prod}
\mathbb{A}_1/t^2\mathbb{A}_1 \times_{\mathcal{O}_{X/\mathcal{P}}} \mathbb{A}_0/t^2 \mathbb{A}_0 \to \shEnd_\mathbb{C}(\gr\mathbb{A}_{01})
\end{equation}

The sheaf $\mathbb{A}_1/t^2\mathbb{A}_1\ominus\mathbb{A}_0/t^2 \mathbb{A}_0$ (the Baer difference of extensions of $\mathcal{O}_{X/\mathcal{P}}$ by $\mathcal{O}_{X/\mathcal{P}}$) is defined by the push-out square
\[
\begin{CD}
\mathcal{O}_{X/\mathcal{P}}\times\mathcal{O}_{X/\mathcal{P}} @>{(\cdot t)\times(\cdot t)}>> \mathbb{A}_1/t^2\mathbb{A}_1 \times_{\mathcal{O}_{X/\mathcal{P}}} \mathbb{A}_0/t^2 \mathbb{A}_0 \\
@VVV @VVV \\
\mathcal{O}_{X/\mathcal{P}} @>>> \mathbb{A}_1/t^2\mathbb{A}_1\ominus\mathbb{A}_0/t^2 \mathbb{A}_0
\end{CD}
\]
where the left vertical map is defined by $(f_1,f_0) \mapsto f_1 - f_0$. Here, for a DQ-algebra $\mathbb{A}$ we use the identification $\mathcal{O}_{X/\mathcal{P}} \cong t\mathbb{A}/t^2\mathbb{A}$ given by the multiplication by $t$.

The sheaf $\mathbb{A}_1/t^2\mathbb{A}_1 \times_{\mathcal{O}_{X/\mathcal{P}}} \mathbb{A}_0/t^2 \mathbb{A}_0$ has a canonical structure of a Lie algebra as a Lie subalgebra of $\mathbb{A}_1/t^2\mathbb{A}_1 \times \mathbb{A}_0/t^2 \mathbb{A}_0$ with the bracket $\widetilde{\{\ ,\ \}}$ (see \ref{subsection: Subprincipal symbols}) on each factor. The bracket on $\mathbb{A}_1/t^2\mathbb{A}_1 \times_{\mathcal{O}_{X/\mathcal{P}}} \mathbb{A}_0/t^2 \mathbb{A}_0$ descends to a Lie bracket on $\mathbb{A}_1/t^2\mathbb{A}_1\ominus\mathbb{A}_0/t^2 \mathbb{A}_0$.

\begin{prop}\label{prop: baer to atiyah}
{~}
\begin{enumerate}
\item The map \eqref{action of fib prod} is a morphism of Lie algebras.

\item The map \eqref{action of fib prod} factors through $\mathbb{A}_1/t^2\mathbb{A}_1\ominus\mathbb{A}_0/t^2 \mathbb{A}_0$.

\item The induced map
\begin{equation*}
\mathbb{A}_1/t^2\mathbb{A}_1 \ominus \mathbb{A}_0/t^2 \mathbb{A}_0 \to \shEnd_\mathbb{C}(\gr\mathbb{A}_{01})
\end{equation*}
takes values in the Atiyah algebra $\mathcal{A}_{\gr\mathbb{A}_{01}}$ of the line bundle $\gr\mathbb{A}_{01}$ (see \ref{subsubsection: Atiyah algebras}).
\item The diagram
\begin{equation*}
\begin{CD}
\mathcal{O}_{X/\mathcal{P}} @>>> \mathbb{A}_1/t^2\mathbb{A}_1\ominus\mathbb{A}_0/t^2 \mathbb{A}_0 @>>>\mathcal{O}_{X/\mathcal{P}} \\
@V{f\mapsto f\cdot id}VV @VVV @VV{d_\Pi}V \\
\shEnd_{\mathcal{O}_{X/\mathcal{P}}}(\gr\mathbb{A}_{01}) @>>> \mathcal{A}_{\gr\mathbb{A}_{01}} @>>> \mathcal{T}_{X/\mathcal{P}}
\end{CD}
\end{equation*}
is commutative.
\end{enumerate}
\end{prop}
\begin{proof}
For $(a,a^\prime) \in \mathbb{A}_1\times_{\mathcal{O}_{X/\mathcal{P}}}\mathbb{A}_0$ and $b\in \mathbb{A}_{01}$ let $[(a,a^\prime),b] := ab - ba^\prime$.
\begin{enumerate}
\item For $(a_i,a^\prime_i) \in \mathbb{A}_1\times_{\mathcal{O}_{X/\mathcal{P}}}\mathbb{A}_0$, $i=1,2$ and $b\in \mathbb{A}_{01}$
\begin{multline*}
[(a_1,a^\prime_1), [(a_2,a^\prime_2), b]] - [(a_2,a^\prime_2), [(a_1,a^\prime_1), b]] = \\
[(a_1,a^\prime_1), a_2b - ba^\prime_2] - [(a_2,a^\prime_2),a_1b - ba^\prime_1] = \\
a_1a_2b - a_2ba^\prime_1 -  a_1ba^\prime_2 + ba^\prime_2a^\prime_1 - a_2a_1b + a_1ba^\prime_2 + a_2ba^\prime_1 - ba^\prime_1a^\prime_2 = \\
[a_1, a_2]b - b[a^\prime_1, a^\prime_2] = [([a_1, a_2],[a^\prime_1, a^\prime_2]),b]  ,
\end{multline*}
which implies the claim.

\item The composition
\begin{equation*}
\left(\mathcal{O}_{X/\mathcal{P}}\times\mathcal{O}_{X/\mathcal{P}}\right) \otimes \gr\mathbb{A}_{01} \to \mathbb{A}_1/t^2\mathbb{A}_1\times_{\mathcal{O}_{X/\mathcal{P}}}\mathbb{A}_0/t^2 \mathbb{A}_0 \otimes \gr\mathbb{A}_{01} \to \gr\mathbb{A}_{01}
\end{equation*}
is given by $(f_1,f_0)\otimes b \mapsto (f_1 - f_0)\cdot b$. This means that the map \eqref{action of fib prod} factors through $\mathbb{A}_1/t^2\mathbb{A}_1\ominus\mathbb{A}_0/t^2 \mathbb{A}_0$.

\item 
\[
[(a,a^\prime),fb] - f[(a,a^\prime),fb] = [a,f]b = t\{\sigma(a), f\}b + t^2\mathbb{A}_{01}
\]
which shows that $\mathbb{A}_1/t^2\mathbb{A}_1\ominus\mathbb{A}_0/t^2 \mathbb{A}_0$ acts by differential operators of order one.

\item The same calculation shows that the principal symbol of the operator $b \mapsto [(a,a'),b]$ is given by $d_\Pi\sigma(a)$.

\end{enumerate}
\end{proof}

Let $\widetilde{\Pi}_{\gr\mathbb{A}_{01}} = \Pi\times_{\mathcal{T}_{X/\mathcal{P}}}\mathcal{A}_{\gr\mathbb{A}_{01}}$ denote the $\mathcal{O}_{X/\mathcal{P}}$-extension of $\Pi$ obtained by pull-back by the anchor map $\widetilde{\poisson}$. Thus, we have the commutative diagram
\begin{equation*}
\begin{CD}
\mathcal{O}_{X/\mathcal{P}} @>>> \mathbb{A}_1/t^2\mathbb{A}_1\ominus\mathbb{A}_0/t^2 \mathbb{A}_0 @>>>\mathcal{O}_{X/\mathcal{P}} \\
@| @VVV @VV{d}V \\
\mathcal{O}_{X/\mathcal{P}} @>>> \widetilde{\Pi}_{\gr\mathbb{A}_{01}} @>>> \Pi
\end{CD}
\end{equation*}
For $\phi_i \in \Sigma_1(\mathbb{A}_i)$, $i=0,1$, the composition
\[
\mathcal{O}_{X/\mathcal{P}} \xrightarrow{(\phi_1,\phi_0)} \mathbb{A}_1/t^2\mathbb{A}_1\ominus\mathbb{A}_0/t^2 \mathbb{A}_0 \to \widetilde{\Pi}_{\gr\mathbb{A}_{01}}
\]
is a derivation, hence factors uniquely as
\[
\mathcal{O}_{X/\mathcal{P}} \xrightarrow{d} \Pi \xrightarrow{\nabla_{\phi_1,\phi_0}} \widetilde{\Pi}_{\gr\mathbb{A}_{01}}
\]
with $\nabla_{\phi_1,\phi_0} \in \CONN_\Pi(\widetilde{\Pi}_{\gr\mathbb{A}_{01}})$.

The assignment $(\phi_1,\phi_0) \mapsto \nabla_{\phi_1,\phi_0}$ defines the map
\begin{equation}\label{diff Sigma to conn}
\nabla(\mathbb{A}_{01}) \colon \Sigma_1(\mathbb{A}_1)\ominus\Sigma_1(\mathbb{A}_0) \to \CONN_\Pi(\widetilde{\Pi}_{\gr\mathbb{A}_{01}}) .
\end{equation}

The proof of the following lemma is left to the reader.

\begin{lemma}\label{lemma: nabla map of bimodules}
Suppose that $\Psi \colon \mathbb{A}_{01}^{(1)} \to \mathbb{A}_{01}^{(0)}$ is an isomorphism of bi-invertible $\mathbb{A}_1\otimes_{\mathbb{C}[[t]]}\mathbb{A}_0^\op$-modules. Then $\nabla(\mathbb{A}_{01}^{(0)}) = \nabla(\mathbb{A}_{01}^{(1)}) \circ \Conn_\Pi(\Ad(\Psi))$
\end{lemma}

\begin{prop}\label{prop: inv sec to conn is gr dlog}
Suppose that $\mathbb{A}_{01}$ is a bi-invertible $\mathbb{A}_1\otimes_{\mathbb{C}[[t]]}\mathbb{A}_0^\op$-module.
\begin{enumerate}
\item The map $\nabla(\mathbb{A}_{01})$ is a morphism of $(\Omega^1_\Pi \xrightarrow{d_\Pi} \Omega^{2,cl}_\Pi)$-torsors.

\item The diagram
\[
\begin{CD}
(\gr\mathbb{A}_{01})^\times @>{d_\Pi\log}>> \CONN_\Pi(\widetilde{\Pi}_{\gr\mathbb{A}_{01}}) \\
@V{\Sigma_1(\mathbb{A}_{01})}VV @AA{\nabla(\mathbb{A}_{01})}A \\
\shHom_{(\Omega_\Pi^1 \xrightarrow{d_\Pi} \Omega_\Pi^{2,cl})[1]}(\Sigma_1(\mathbb{A}_0), \Sigma_1(\mathbb{A}_1)) @>>> \Sigma_1(\mathbb{A}_1)\ominus\Sigma_1(\mathbb{A}_0)
\end{CD}
\]
is commutative.

\end{enumerate}
\end{prop}
\begin{proof}
{~}
\begin{enumerate}
\item Since the curvature is $\mathcal{O}_{X/\mathcal{P}}$-bilinear it is sufficient to check that $c(\nabla_{\phi_1,\phi_0}) = c(\phi_1) - c(\phi_0)$ on exact forms. The latter identity follows from Proposition \ref{prop: baer to atiyah}.

\item Recall (\eqref{dlog L}, see \ref{subsubsection: invertible B-modules}) that the upper horizontal map $d_\Pi\log$
is characterized as follows. A section $b\in (\gr\mathbb{A}_{01})^\times$ establishes a morphism of $\mathcal{O}_{X/\mathcal{P}}^\times$-torsors $\tau_b \colon\mathcal{O}_{X/\mathcal{P}}^\times \to (\gr\mathbb{A}_{01})^\times$ which induces the morphism of 
$\Omega^1_\Pi$-torsors $\Omega^1_\Pi \to \CONN_\Pi(\widetilde{\Pi}_{\gr\mathbb{A}_{01}})$ given by $\alpha \mapsto \nabla_b +\alpha$, where $\nabla_b$ is the connection induced on $(\gr\mathbb{A}_{01})^\times$ by the canonical flat connection on $\mathcal{O}_{X/\mathcal{P}}^\times$. The map $d_\Pi\log$ is the unique map making the diagram
\[
\begin{CD}
\mathcal{O}_{X/\mathcal{P}}^\times @>{d_\Pi\log}>> \Omega^1_\Pi \\
@V{\tau_b}VV @VV{\alpha \mapsto \nabla_b +\alpha}V \\
(\gr\mathbb{A}_{01})^\times @>{d_\Pi\log}>> \CONN_\Pi(\widetilde{\Pi}_{\gr\mathbb{A}_{01}})
\end{CD}
\]
commutative. Therefore, it suffices to show that the composition $\nabla(\mathbb{A}_{01}) \circ \Sigma_1(\mathbb{A}_{01})$ has the same property.

A section $b\in\mathbb{A}_{01}^\times$ determines the isomorphism $\Phi_b \colon \mathbb{A}_0 \to \mathbb{A}_1$  and a compatible
isomorphism of $\mathbb{A}_{01}$ with $\mathbb{A}_1$ viewed as $\mathbb{A}_1\otimes_{\mathbb{C}[[t]]}\mathbb{A}_1^\op$-module.
Functoriality of $\Sigma_1$, Lemma \ref{lemma: Sigma 1 map of bimodules} and Lemma \ref{lemma: nabla map of bimodules} show that it is sufficient to prove the statement in the case when $\mathbb{A}_0 = \mathbb{A}_1= \mathbb{A}$, and $\mathbb{A}_{01}= \mathbb{A}$ is the ``diagonal" $\mathbb{A}$-bi-module corresponding to $b = 1 \in \mathbb{A}$.

Suppose that $\mathbb{A}_0 = \mathbb{A}_1 = \mathbb{A}$ and $\mathbb{A}_{01} = \mathbb{A}$. In this case $\shHom_{(\Omega_\Pi^1 \xrightarrow{d_\Pi} \Omega_\Pi^{2,cl})[1]}(\Sigma_1(\mathbb{A}), \Sigma_1(\mathbb{A})) = \Omega^{1,cl}_\Pi$, $\Sigma_1(\mathbb{A}) = d_\Pi\log \colon \mathcal{O}_{X/\mathcal{P}}^\times \to \Omega^{1,cl}_\Pi$ while the map $\nabla(\mathbb{A}) \colon \Omega^{1,cl}_\Pi \hookrightarrow \Omega^1_\Pi$ is the inclusion, which proves the claim.
\end{enumerate}
\end{proof}

Suppose that $\mathbb{A}_i$, $i=0,1,2$, is a DQ-algebra and $\mathbb{A}_{i,i+1}$, $i=0,1$ is a bi-invertible $\mathbb{A}_{i+1}\otimes_{\mathbb{C}[[t]]}\mathbb{A}_i^\op$-module. Thus, all DQ-algebras $\mathbb{A}_i$ give rise to the same associated Poisson bi-vector $\poisson$, hence the same Lie algebroid $\Pi$.

Let $\mathbb{A}_{02} := \mathbb{A}_{01}\otimes_{\mathbb{A}_1}\mathbb{A}_{12}$. Then, $\mathbb{A}_{02}$ is a bi-invertible $\mathbb{A}_2\otimes_{\mathbb{C}[[t]]}\mathbb{A}_0^\op$-module with $\gr\mathbb{A}_{02} = \gr\mathbb{A}_{01}\otimes_{\mathcal{O}_{X/\mathcal{P}}}\gr\mathbb{A}_{12}$.

The bi-linear paring $\gr\mathbb{A}_{01}\times \gr\mathbb{A}_{12} \to \gr\mathbb{A}_{02}$ gives rise to the map  $(\gr\mathbb{A}_{01})^\times \times (\gr\mathbb{A}_{12})^\times \to (\gr\mathbb{A}_{02})^\times$.

\begin{prop}\label{prop: sigma one composition}
The diagram
\[
\begin{CD}
(\gr\mathbb{A}_{01})^\times \times (\gr\mathbb{A}_{12})^\times @>{\Sigma_1(\mathbb{A}_{01})\times\Sigma_1(\mathbb{A}_{12})}>> \shHom(\Sigma_1(\mathbb{A}_0),\Sigma_1(\mathbb{A}_1))\times \shHom(\Sigma_1(\mathbb{A}_1),\Sigma_1(\mathbb{A}_2)) \\
@V{\otimes}VV @VV{\circ}V \\
(\gr\mathbb{A}_{02})^\times @>{\Sigma_1(\mathbb{A}_{02})}>> \shHom(\Sigma_1(\mathbb{A}_0),\Sigma_1(\mathbb{A}_2))
\end{CD}
\]
is commutative.
\end{prop}
\begin{proof}
The commutativity of the diagram
\[
\begin{CD}
(\gr\mathbb{A}_{01})^\times \times (\gr\mathbb{A}_{12})^\times @>>> \shHom_{\DQA}(\mathbb{A}_0, \mathbb{A}_1)\times\shHom_{\DQA}(\mathbb{A}_1, \mathbb{A}_2) \\
@V{\otimes}VV @VV{\circ}V \\
(\gr\mathbb{A}_{02})^\times @>>> \shHom_{\DQA}(\mathbb{A}_0, \mathbb{A}_2)
\end{CD}
\]
with the horizontal maps as in \eqref{inv bimod to hom alg} is left to the reader. The rest follows from functorial properties of $\Sigma_1$.
\end{proof}

\section{DQ-algebroids}\label{section: qcl}

\subsection{DQ-algebroids}
Recall the definition of DQ-algebroids from \cite{KS}.

\begin{definition}[\cite{KS}, Definition 2.3.1]
A \emph{DQ-algebroid} $\mathcal{C}$ is a $\mathbb{C}[[t]]$-algebroid such that for each open set $U\subseteq X$ with $\mathcal{C}(U) \neq \varnothing$ and any $L\in \mathcal{C}(U)$ the $\mathbb{C}[[t]]$-algebra $\shEnd_\mathcal{C}(L)$ is a DQ-algebra on $U$.
\end{definition}

In other words, a DQ-algebroid is a $\mathbb{C}[[t]]$-algebroid locally equivalent to a star-product.

\begin{lemma}\label{lemma: hom bi-invertible}
Suppose that $\mathcal{C}$ is a DQ-algebroid, $U\subseteq X$ is an open subset such that $\mathcal{C}(U) \neq \varnothing$. For $L,L^\prime \in \mathcal{C}(U)$ the $\shEnd_\mathcal{C}(L^\prime)\otimes\shEnd_\mathcal{C}(L)^\op$-module $\shHom_\mathcal{C}(L,L^\prime)$ is bi-invertible. Moreover, $\shHom_\mathcal{C}(L,L^\prime)^\times$ coincides with the sub-sheaf of isomorphisms, i.e. $\shHom_{i\mathcal{C}}(L,L^\prime) = \shHom_\mathcal{C}(L,L^\prime)^\times$.
\end{lemma}
\begin{proof}
Left to the reader.
\end{proof}

For a DQ-algebroid $\mathcal{C}$ we denote by $\mathcal{C}/t\cdot\mathcal{C}$ the separated prestack with the same objects as $\mathcal{C}$ and $\Hom_{\mathcal{C}/t\cdot\mathcal{C}}(L_1,L_2) = \Gamma(U;\gr\shHom_\mathcal{C}(L_1,L_2))$ for $L_1,L_2\in\mathcal{C}(U)$ and by  $\gr\mathcal{C}$ the associated stack. We denote by $\gr\colon \mathcal{C} \to \gr\mathcal{C}$ the ``principal symbol" functor.

For $U\subseteq X$ let $\DQ_{X/\mathcal{P}}(U)$ denote the 2-category of DQ-algebroids.The assignment $U \mapsto \DQ_{X/\mathcal{P}}(U)$ extends to a 2-stack which we denote $\DQ_{X/\mathcal{P}}$.
%

The assignment $\mathcal{C} \mapsto \gr^\times\mathcal{C} := i\gr\mathcal{C}$ gives rise to the morphism of 2-stacks
\begin{equation}\label{gr DQ}
\gr^\times\colon \DQ_{X/\mathcal{P}} \to \mathcal{O}^\times_{X/\mathcal{P}}[2] .
\end{equation}

\subsection{The associated Poisson structure}\label{subsection: The associated Poisson structure}
The following proposition is an immediate consequence of Lemma \ref{lemma: hom bi-invertible} and Lemma \ref{lemma: bi-inv same poisson}.
\begin{prop}
There exists a unique Poisson bracket
\begin{equation*}
\{.,.\}^\mathcal{C} \colon \mathcal{O}_{X/\mathcal{P}}\otimes\mathcal{O}_{X/\mathcal{P}} \to \mathcal{O}_{X/\mathcal{P}}
\end{equation*}
such that for any $U\subset X$ with $\mathcal{C}(U)\neq\varnothing$ and any $L\in\mathcal{C}(U)$ the restriction of $\{.,.\}^\mathcal{C}$ to $U$ coincides with the Poisson bracket associated to the DQ-algebra $\shEnd_\mathcal{C}(L)$.
\end{prop}

\begin{notation}
We denote by $\poisson^\mathcal{C} \in \Gamma(X;\bigwedge^2\mathcal{T}_{X/\mathcal{P}})$ the Poisson bi-vector which corresponds to $\{.,.\}^\mathcal{C}$. The Poisson bi-vector $\poisson^\mathcal{C}$ gives rise to a structure of a Lie algebroid on $\Omega^1_{X/\mathcal{P}}$ as described in \ref{subsubsection: Poisson structures}. We denote this Lie algebroid by $\Pi^\mathcal{C}$. \qed
\end{notation}

The assignment $\mathcal{C} \mapsto \poisson^\mathcal{C}$ give rise to the canonical morphism
\begin{equation}\label{map: assoc poisson}
\DQ_{X/\mathcal{P}} \to \Lambda^2\mathcal{T}_{X/\mathcal{P}}
\end{equation}
For $\pi\in\Gamma(X;\Lambda^2\mathcal{T}_{X/\mathcal{P}})$ we denote by $\DQ^\pi_{X/\mathcal{P}}$ the fiber of \eqref{map: assoc poisson} over $\pi$.

\subsection{The associated $\Pi^\mathcal{C}$-connective structure}
Suppose that $\mathcal{C}$ is a DQ-algebroid.

For $U\subseteq X$ an open subset such that $\mathcal{C}(U) \neq \varnothing$, $L,L^\prime \in \mathcal{C}(U)$ the $\shEnd_\mathcal{C}(L^\prime)\otimes\shEnd_\mathcal{C}(L)^\op$-module $\shHom_\mathcal{C}(L,L^\prime)$ is bi-invertible by Lemma \ref{lemma: hom bi-invertible}. An isomorphism $\phi\colon L \to L^\prime$ induces the morphism of DQ-algebras $\Ad(\phi) \colon \shEnd_\mathcal{C}(L) \to \shEnd_\mathcal{C}(L^\prime)$ defined by $\psi \mapsto \phi\circ\psi\circ\phi^{-1}$.

The proof of the following lemma is left to the reader.
\begin{lemma}\label{lemma: ad is inv bimod to hom alg}
The map $\Ad\colon \shHom_\mathcal{C}(L,L^\prime)^\times \to \shHom_{\DQA}(\shEnd_\mathcal{C}(L),\shEnd_\mathcal{C}(L^\prime))$ coincides with the map \eqref{inv bimod to hom alg}.
\end{lemma}

For $U\subseteq X$ an open subset such that $\mathcal{C}(U) \neq \varnothing$, $L \in \mathcal{C}(U)$ the $\mathbb{C}[[t]]$-algebra $\shEnd_\mathcal{C}(L)$ is a DQ-algebra and we set
\[
\Conn^\mathcal{C}(L) := \Sigma_1(\shEnd_\mathcal{C}(L)) \in (\Omega^1_{\Pi^\mathcal{C}} \xrightarrow{d_\Pi} \Omega^{2,cl}_{\Pi^\mathcal{C}})[1] .
\]

For $U\subseteq X$ an open subset such that $\mathcal{C}(U) \neq \varnothing$, $L,L^\prime \in \mathcal{C}(U)$, let
\begin{multline*}
\Conn^\mathcal{C}(L,L^\prime) := \Sigma_1(\shHom_\mathcal{C}(L,L^\prime)) \colon \shHom_{\mathcal{O}^\times_{X/\mathcal{P}}[1]}(\gr(L)^\times,\gr(L^\prime)^\times) \\
= (\gr\shHom_\mathcal{C}(L,L^\prime))^\times \to \shHom_{(\Omega^1_{\Pi^\mathcal{C}} \xrightarrow{d_\Pi} \Omega^{2,cl}_{\Pi^\mathcal{C}})[1]}(\Conn^\mathcal{C}(L), \Conn^\mathcal{C}(L^\prime)) .
\end{multline*}

By Proposition \ref{prop: sigma one composition}  the above assignments define a functor, denoted
\[
\Conn^\mathcal{C} \colon i(\mathcal{C}/t\mathcal{C}) \to (\Omega^1_{\Pi^\mathcal{C}} \xrightarrow{d_\Pi} \Omega^{2,cl}_{\Pi^\mathcal{C}})[1] .
\]
As the target is a stack, this functor induces the functor
\[
\Conn^\mathcal{C} \colon i\gr\mathcal{C} \to (\Omega^1_{\Pi^\mathcal{C}} \xrightarrow{d_\Pi} \Omega^{2,cl}_{\Pi^\mathcal{C}})[1] ,
\]
i.e. a $\Pi^\mathcal{C}$-connective structure with flat curving on $i\gr\mathcal{C}$.

The assignment $\mathcal{C} \mapsto (\gr\mathcal{C}, \Conn^\mathcal{C})$ defines the morphism of 2-stacks
\begin{equation}\label{morphism: qcl}
\widetilde{\gr^\times}\colon \DQ^\pi_{X/\mathcal{P}} \to (\mathcal{O}^\times_{X/\mathcal{P}} \to \Omega^1_\Pi \to \Omega^{2,cl}_\Pi)[2]
\end{equation}
(see \ref{subsection: The associated Poisson structure} and \ref{subsection: curving} for notation) lifting the morphism \eqref{gr DQ}. We refer to \eqref{morphism: qcl} as the morphism of quasi-classical limit

\subsection{Obstruction to quantization}
The construction of the canonical Poisson connective structure with flat curving on the classical limit of a DQ-algebroid imposes restrictions on the class of the classical limit. We will say that a  twisted form $\mathcal{S}$ of $\mathcal{O}_{X/\mathcal{P}}$ admits a deformation along a Poisson bi-vector $\pi \in \Gamma(X;\bigwedge^2\mathcal{T}_{X/\mathcal{P}})$ if there exists a DQ-algebroid $\mathcal{C}$ with $\pi^\mathcal{C} = \pi$ and $\gr\mathcal{C}\cong\mathcal{S}$.

\begin{thm}\label{thm: obst and gemuese}
A twisted form $\mathcal{S}$ of $\mathcal{O}_{X/\mathcal{P}}$ admits a deformation along a Poisson bi-vector $\pi \in \Gamma(X;\bigwedge^2\mathcal{T}_{X/\mathcal{P}})$ \emph{only} if the class of $\mathcal{S}$ in $H^2(X;\mathcal{O}^\times_{X/\mathcal{P}})$ is in the image of the map
\begin{equation}\label{map on H2}
H^2(X;\mathcal{O}^\times_{X/\mathcal{P}} \to \Omega^1_\Pi \to \Omega^{2,cl}_\Pi ) \to H^2(X;\mathcal{O}^\times_{X/\mathcal{P}}) ,
\end{equation}
where $\Pi$ denotes the Lie algebroid associated with $\pi$.
\end{thm}
\begin{proof}
The map \eqref{map on H2} is induced by the map of complexes of sheaves
\[
(\mathcal{O}^\times_{X/\mathcal{P}} \to \Omega^1_\Pi \to \Omega^{2,cl}_\Pi) \to \mathcal{O}^\times_{X/\mathcal{P}} .
\]
The induced morphism of 2-stacks
\[
(\mathcal{O}^\times_{X/\mathcal{P}} \to \Omega^1_\Pi \to \Omega^{2,cl}_\Pi)[2] \to \mathcal{O}^\times_{X/\mathcal{P}}[2]
\]
is the functor of forgetting the $\Pi$-connective structure given by the $(\mathcal{S}, \kappa) \mapsto \mathcal{S}$.

If there exists a DQ-algebroid $\mathcal{C}$ with the associated Poisson bi-vector $\pi$ and $\gr^\times\mathcal{C}$ equivalent to $\mathcal{S}$, then the class of $\widetilde{\gr^\times}\mathcal{C}$ in $H^2(X;\mathcal{O}^\times_{X/\mathcal{P}} \to \Omega^1_\Pi \to \Omega^{2,cl}_\Pi )$ is a lift of the class of $\mathcal{S}$ in $H^2(X;\mathcal{O}^\times_{X/\mathcal{P}})$.
\end{proof}

The following corollary of Theorem \ref{thm: obst and gemuese} is well-known (see for example \cite{BGNT0}).
\begin{cor}\label{cor: symplectic dq}
A twisted form $\mathcal{S}$ of $\mathcal{O}_{X/\mathcal{P}}$ admits a symplectic deformation only if the class of $\mathcal{S}$ in $H^2(X;\mathcal{O}^\times_{X/\mathcal{P}})$ is in the image of the map $H^2(X;\mathbb{C}^\times) \to H^2(X;\mathcal{O}^\times_{X/\mathcal{P}})$.
\end{cor}
\begin{proof}
If the associated Poisson bi-vector is non-degenerate, then the anchor map $\widetilde{\pi}$ is an isomorphism and the complex $\mathcal{O}^\times_{X/\mathcal{P}} \to \Omega^1_\Pi \to \Omega^{2,cl}_\Pi$ is isomorphic to the the complex $\mathcal{O}^\times_{X/\mathcal{P}} \to \Omega^1_{X/\mathcal{P}} \to \Omega^{2,cl}_{X/\mathcal{P}}$. The Poincar\'e Lemma holds for $X/\mathcal{P}$, in other words, the map (inclusion of locally constant functions) $\mathbb{C}^\times_X \to (\mathcal{O}^\times_{X/\mathcal{P}} \to \Omega^1_{X/\mathcal{P}} \to \Omega^{2,cl}_{X/\mathcal{P}})$ is a quasi-isomorphism.
\end{proof}

\begin{remark}
In fact, as shown in \cite{BGNT0}, the converse of Corollary \ref{cor: symplectic dq} holds.
\end{remark}

\section{The quasi-classical limit and formality for algebroids}\label{section: formality}
In this section we formulate a conjecture relating the quasi-classical limit with the results of \cite{BGNT1} which describe the formal deformation theory of gerbes in terms of quasi-classical data.
For simplicity we will assume that the distribution $\mathcal{P}$ is trivial, i.e. we are dealing with a plain $C^\infty$ manifold.

\subsection{The de Rham class of an $\mathcal{O}^\times_X$-gerbe}
We recall the construction of a closed differential 3-form representing the class of an $\mathcal{O}^\times_X$-gerbe. See \cite{JLB} for further details.

Suppose that $\mathcal{S}$ is a twisted form of $\mathcal{O}_X$. Under the map $H^2(X;\mathcal{O}_X^\times) \xrightarrow{d\log} H^2(X;\Omega^{1,cl}_X) \cong H^3_{dR}(X)$ the equivalence class $[\mathcal{S}]\in H^2(X;\mathcal{O}_X^\times)$ is mapped to the class $[\mathcal{S}]_{dR}\in H^3_{dR}(X)$. We briefly recall the construction of a representative $H\in\Gamma(X;\Omega_X^{3,cl})$ of the class $[\mathcal{S}]_{dR}$.

The map $H^2(X;\mathcal{O}_X^\times \to \Omega^1_X \to \Omega^2_X) \to H^2(X;\mathcal{O}_X^\times)$ is surjective, in other words, every $\mathcal{O}_X^\times$-gerbe admits a connective structure with curving (a.k.a. a classical $B$-field). The de Rham differential gives rise to the map of complexes of sheaves $d \colon (\mathcal{O}_X^\times \to \Omega^1_X \to \Omega^2_X)[2] \to \Omega^{3,cl}_X$. The induced map on cohomology $d \colon H^2(X;\mathcal{O}_X^\times \to \Omega^1_X \to \Omega^2_X) \to \Gamma(X;\Omega^{3,cl}_X)$ corresponds to the map which associates to a gerbe with connective structure with curving  $(\mathcal{S}, \kappa)$ a closed 3-form $c(\kappa)$ called \emph{the curvature of the curving $\kappa$}. The class of $c(\kappa)$ in $H^3_{dR}(X)$ depends only on the equivalence class of $\mathcal{S}$ (and does not depend on the choice of the connective structure) and is denoted $[\mathcal{S}]_{dR}\in H^3_{dR}(X)$. 

The map $H^2(X;\mathcal{O}_X^\times) \to H^3_{dR}(X)$ given by $[\mathcal{S}] \mapsto [\mathcal{S}]_{dR}$ coincides with the composition $H^2(X;\mathcal{O}_X^\times) \to H^2(X;\Omega)X^{1,cl}) \cong H^3_{dR}(X)$ induced by the map of sheaves $ d\log \colon \mathcal{O}_X^\times \to \Omega^{1,cl}_X$.

\subsection{Deformations of twisted forms of $\mathcal{O}_X$ (\cite{BGNT1})}\label{subsection: BGNT}
Suppose that $\mathcal{S}$ is a twisted form of $\mathcal{O}_X$. Theorem 6.5 in conjunction with Remark 6.6 of \cite{BGNT1} imply that equivalence classes of DQ-algebroids with associated Poisson bi-vector $\pi$ and classical limit (equivalent to) $\mathcal{S}$ are in bijective correspondence with equivalence classes of pairs $(\pi_t, H)$, where
\begin{itemize}
\item $\pi_t = 0 + \pi_1 t + \pi_2 t^2 + \cdots \in \Gamma(X; \bigwedge^2\mathcal{T}_X\widehat{\otimes} t\mathbb{C}[[t]])$

\item $H \in \Gamma(X;\Omega^{3,cl}_X)$
\end{itemize}
satisfying
\begin{enumerate}
\item $\pi_1 = \pi$

\item $H$ is a representative of $[\mathcal{S}]_{dR}\in H^3_{dR}(X)$

\item $[\pi_t,\pi_t] = \widetilde{\pi_t}^{\wedge 3}(H)$, 
\end{enumerate}
where $\widetilde{\pi_t} \colon \Omega_X^1 \to \mathcal{T}_X \widehat{\otimes} t\mathbb{C}[[t]]$ is the adjoint of $\pi_t$.

\subsection{$\Pi$-connective structures from quasi-classical data}\label{subsection: Pi-conn from MC}
Suppose that $\pi$ is a Poisson bi-vector on $X$ with associated Lie algebroid denoted $\Pi$ and $(\mathcal{S},\nabla)$ is gerbe equipped with a connective structure $\Conn$ with curving whose curvature is equal to $H \in \Gamma(X;\Omega^{3,cl}_X)$.

Given $\pi_t \in \Gamma(X; \bigwedge^2\mathcal{T}_X \widehat{\otimes} t\mathbb{C}[[t]])$ such that the pair $(\pi_t, H)$ satisfies the conditions in \ref{subsection: BGNT} one obtains a canonical $\Pi$-connective structure with flat curving on $\mathcal{S}$ as follows.

To a locally defined object $L\in\mathcal{S}$ the connective structure $\Conn$ associates a $\Omega^1_X$-torsor $\Conn(L)$ equipped with the curvature map $c_L\colon \Conn(L) \to \Omega^2_X$. As $H$ is the curvature of the curving of $\nabla$, the composition $\Conn(L) \xrightarrow{c_L} \Omega^2_X \xrightarrow{d} \Omega^{3,cl}_X$ is constant and equal to $H$.

The $\Omega^1_X$-torsor $\Conn(L)$ gives rise to the $\Omega^1_\Pi$-torsor $\widetilde{\poisson}\Conn(L)$.  The map $(\widetilde{\poisson}(c_L)-\pi_2) \colon \widetilde{\poisson}\Conn(L) \to \Omega^2_\Pi$ (where $\pi_2\in\Omega^2_\Pi$ is viewed as a constant map) is a morphism of torsors compatible with the map of groups $d_\Pi \colon \Omega^1_\Pi \to \Omega^2_\Pi$.

The condition $[\pi_t,\pi_t] = \widetilde{\pi_t}^{\wedge 3}(H)$ implies that $[\pi,\pi_2] = \widetilde{\pi}^{\wedge 3}(H)$. Therefore, the composition $\widetilde{\pi}\Conn(L) \xrightarrow{\widetilde{\pi}(c_L)-\pi_2} \Omega^2_\Pi \xrightarrow{d_\Pi} \Omega^3_\Pi$ is equal to zero. Consequently, the map $\widetilde{\poisson}(c_L)-\pi_2$ takes values in $\Omega^{2,cl}_\Pi$ and, hence, the assignment $\mathcal{S}\ni L \mapsto (\widetilde{\poisson}\Conn(L), \widetilde{\poisson}(c_L)-\pi_2)$ defines a $\Pi$-connective structure with flat curving on $\mathcal{S}$.

\subsection{Quasi-classical limit and formality}
Suppose that $\mathcal{C}$ is a DQ-algebroid. Let $(\pi_t, H)$ be a pair representing the equivalence class of $\mathcal{C}$ under the bijection (see \ref{subsection: BGNT}) induced by \emph{a choice of a formality isomorphism}.
\begin{conj}
The quasi-classical limit $\widetilde{\gr^\times}\mathcal{C}$ is equivalent to $\gr^\times\mathcal{C}$ equipped with the $\Pi$-connective structure with flat curving deduced from $(\pi_t, H)$ as in \ref{subsection: Pi-conn from MC}. Moreover, this equivalence is independent of the choice of formality isomorphism.
\end{conj}


\begin{thebibliography}{BGNT15}

\bibitem[BCG97]{BCG97}
M{\'e}lanie Bertelson, Michel Cahen, and Simone Gutt.
\newblock Equivalence of star products.
\newblock {\em Classical Quantum Gravity}, 14(1A):A93--A107, 1997.
\newblock Geometry and physics.

\bibitem[BGKP16]{BGKP}
Vladimir Baranovsky, Victor Ginzburg, Dmitry Kaledin, and Jeremy Pecharich.
\newblock Quantization of line bundles on lagrangian subvarieties.
\newblock {\em Selecta Math. (N.S.)}, 22(1):1--25, 2016.

\bibitem[BGNT07]{BGNT0}
Paul Bressler, Alexander Gorokhovsky, Ryszard Nest, and Boris Tsygan.
\newblock Deformation quantization of gerbes.
\newblock {\em Adv. Math.}, 214(1):230--266, 2007.

\bibitem[BGNT15]{BGNT1}
Paul Bressler, Alexander Gorokhovsky, Ryszard Nest, and Boris Tsygan.
\newblock Formality theorem for gerbes.
\newblock {\em Adv. Math.}, 273:215--241, 2015.

\bibitem[Bry93]{JLB}
Jean-Luc Brylinski.
\newblock {\em Loop spaces, characteristic classes and geometric quantization},
  volume 107 of {\em Progress in Mathematics}.
\newblock Birkh\"auser Boston, Inc., Boston, MA, 1993.

\bibitem[Bur01]{B01}
Henrique Bursztyn.
\newblock Poisson vector bundles, contravariant connections and deformations.
\newblock {\em Progr. Theoret. Phys. Suppl.}, (144):26--37, 2001.
\newblock Noncommutative geometry and string theory (Yokohama, 2001).

\bibitem[Bur02]{B02}
Henrique Bursztyn.
\newblock Semiclassical geometry of quantum line bundles and {M}orita
  equivalence of star products.
\newblock {\em Int. Math. Res. Not.}, (16):821--846, 2002.

\bibitem[BW04]{BW04}
Henrique Bursztyn and Stefan Waldmann.
\newblock Bimodule deformations, {P}icard groups and contravariant connections.
\newblock {\em $K$-Theory}, 31(1):1--37, 2004.

\bibitem[Del73]{D}
Pierre Deligne.
\newblock {\em La formule de dualite globale}, pages 481--587.
\newblock Springer Berlin Heidelberg, Berlin, Heidelberg, 1973.

\bibitem[FW79]{FW}
Hans~R. Fischer and Floyd~L. Williams.
\newblock Complex-foliated structures. {I}. {C}ohomology of the
  {D}olbeault-{K}ostant complexes.
\newblock {\em Trans. Amer. Math. Soc.}, 252:163--195, 1979.

\bibitem[Kas96]{Kash}
Masaki Kashiwara.
\newblock Quantization of contact manifolds.
\newblock {\em Publ. Res. Inst. Math. Sci.}, 32(1):1--7, 1996.

\bibitem[Kon01]{K01}
Maxim Kontsevich.
\newblock Deformation quantization of algebraic varieties.
\newblock {\em Lett. Math. Phys.}, 56(3):271--294, 2001.
\newblock EuroConf{\'e}rence Mosh{\'e} Flato 2000, Part III (Dijon).

\bibitem[Kos70]{Kostant}
Bertram Kostant.
\newblock Quantization and unitary representations. {I}. {P}requantization.
\newblock In {\em Lectures in modern analysis and applications, {III}}, pages
  87--208. Lecture Notes in Math., Vol. 170. Springer, Berlin, 1970.

\bibitem[KS12]{KS}
Masaki Kashiwara and Pierre Schapira.
\newblock Deformation quantization modules.
\newblock {\em Ast\'erisque}, (345):xii+147, 2012.

\bibitem[Mac05]{Mac05}
Kirill C.~H. Mackenzie.
\newblock {\em General theory of {L}ie groupoids and {L}ie algebroids}, volume
  213 of {\em London Mathematical Society Lecture Note Series}.
\newblock Cambridge University Press, Cambridge, 2005.

\bibitem[{Mil}03]{M}
James~S. {Milne}.
\newblock {Gerbes and abelian motives}.
\newblock {\em ArXiv Mathematics e-prints}, January 2003, math/0301304.

\bibitem[Raw77]{Rawnsley}
John~H. Rawnsley.
\newblock On the cohomology groups of a polarisation and diagonal quantisation.
\newblock {\em Trans. Amer. Math. Soc.}, 230:235--255, 1977.

\bibitem[Vey75]{V75}
Jacques Vey.
\newblock D\'eformation du crochet de {P}oisson sur une vari\'et\'e
  symplectique.
\newblock {\em Comment. Math. Helv.}, 50(4):421--454, 1975.

\end{thebibliography}

\end{document}